\documentclass[12pt]{amsart}
\usepackage{hyperref}
\hypersetup{nesting=true,debug=true,naturalnames=true}
\usepackage{graphicx,amssymb,upref}

\usepackage{amsmath,amsthm}
\usepackage{amsfonts}
\usepackage{latexsym}
\usepackage{amsbsy}
\usepackage{amssymb,mathscinet}
\usepackage{tikz-cd}

\numberwithin{equation}{section}

\usepackage{amsfonts,amssymb,amscd,amsmath,enumerate,verbatim,calc,enumerate}
\theoremstyle{plain}
\newtheorem{theorem}{Theorem}[section]
\newtheorem{corollary}[theorem]{Corollary}
\newtheorem{lemma}[theorem]{Lemma}
\newtheorem{proposition}[theorem]{Proposition}

\newtheorem{notation}[theorem]{Notation}
\newtheorem{definition}[theorem]{Definition}

\newtheorem{remark}[theorem]{Remark}

\newcommand{\be}{\mathbb E}

\newcommand{\ot}{\otimes}
\newcommand {\id} {{\textrm{id}}}
\newcommand{\wt}{\widetilde}
\newcommand{\wT}{\wt{T}}

\newcommand{\wV}{\wt{V}}

\begin{document}

\title[]{Berger-Coburn-Lebow representation  for pure isometric representations of product system over $\mathbb N^2_0$}

\date{\today}
	\author[Saini]{Dimple Saini\textsuperscript{*}}
\address{Centre for mathematical and Financial Computing, Department of Mathematics, The LNM Institute of Information Technology, Rupa ki Nangal, Post-Sumel, Via-Jamdoli Jaipur-302031, (Rajasthan) India.}
\email{18pmt006@lnmiit.ac.in, dimple92.saini@gmail.com}
\author[Trivedi]{Harsh Trivedi}
\address{Centre for mathematical and Financial Computing, Department of Mathematics, The LNM Institute of Information Technology, Rupa ki Nangal, Post-Sumel, Via-Jamdoli Jaipur-302031, (Rajasthan) India.}
\email{harsh.trivedi@lnmiit.ac.in, trivediharsh26@gmail.com}
\author[Veerabathiran]{Shankar Veerabathiran}
\address{Indian Statistical Institute, Statistics and Mathematics Unit, 8th Mile, Mysore Road, Bangalore, 560059, India }
\email{shankarunom@gmail.com}


\begin{abstract}
		We obtain Berger-Coburn-Lebow (BCL)-representation for pure isometric covariant representation of product system over $\mathbb{N}_0^2$. Then the corresponding complete set of (joint) unitary invariants is studied, and the BCL- representations are compared with other canonical multi-analytic descriptions of the pure isometric covariant representation. We characterize the invariant subspaces for the pure isometric covariant representation. Also, we study the connection between the joint defect operators and Fringe operators, and the Fredholm index is introduced in this case. Finally, we introduce the notion of congruence relation to classify  the isometric covariant representations of the product system over $\mathbb{N}_0^2$.
\end{abstract}

\keywords{Hilbert $C^*$-modules, covariant representations, Wold decomposition, product systems, Fock space, tuples of operators, wandering subspaces,  isometry, Fredholm operators}
\subjclass[2020]{ 46L08 , 47A15 , 47A53 , 47B38  , 47L55.}

\maketitle

\section{Introduction}
The Wold-von Neumann decomposition (cf. \cite{N29, W38}) of an isometry is a fundamental tool in Operator theory. Berger, Coburn, and Lebow \cite{BCL78} considered $C^*$-algebras generated by $n$ commuting isometries and proposed its Fredholm and representation theory. Yang \cite{Y99} explored the BCL index and related it to the Fredholm tuple (cf. \cite{C81}). The classification problem is one of the important issues in the study of pairs of commuting isometries. Bercovici, Douglas, and Foias \cite{BDF06} classified the pairs of commuting isometries by pivotal operators. He, Qin, and Yang \cite{HQY15} introduced a classification using the congruence relation.

In \cite{BCL78}, a finite set of unitary invariants is derived, corresponding to given $n$ commuting isometries, to show that there exist infinitely many $C^*$-algebras generated by pairs of commuting isometric tuples which are non-isomorphic algebras. Bercovici, Douglas, and Foias \cite{BDF06} gave a full classification of the product of $n$ commuting isometries $V=V_1V_2\ldots V_n,$ when $\mbox{dim}(ker V^*)<\infty.$

Weber \cite{W13} used a deformation of the tensor product and generalized commutation relation of two isometries to show the non-exactness of the $C^*$-algebra, which is a tensor twist of the given isometric pair. In \cite{C77} Cuntz studied $C^*$-algebras generated by row isometries. Popescu \cite{P20} explored doubly $\Lambda$-commuting row isometries and its classification program.

Pimsner \cite{P97} generalized the construction of Cuntz algebras from \cite{C77} using isometric covariant representations. Muhly and Solel \cite{MS99} presented the Wold decomposition for an isometric covariant representation. Arveson started the classification program for $E_0$-semigroups using the (tensor) product system of Hilbert spaces in \cite{A89}. The notion of a discrete product system of Hilbert bi-modules is due to Fowler \cite{F02}. Covariant representation of a product system (see \cite{SZ08, S06}) is an active field of study in operator theory and operator algebras. It provides a unified approach for studying operator tuples on Hilbert spaces in a commuting case on the one hand and the non-commuting case on the other hand, and therefore generalize the setup of Popescu \cite{P20} (for comparison, see \cite{TV21}). 

In this paper, we study BCL-representation for a pure isometric covariant representation of a product system over $\mathbb{N}_0^2$ (see Theorems \ref{BCL} and \ref{K2}) and the classification due to Bercovici, Douglas, and Foias (see Theorems \ref{BCT}), and our approach is based on \cite{BDF06, MSS18, Y99, HQY15}. The section-wise plan is as follows: In Section 1, we recall the basic setup and the classical Wold-von Neumann decomposition for an isometric covariant representation. In Section 2, we obtain the BCL-representation for a pure isometric covariant representation of the product system over $\mathbb{N}_0^2$ and study the connection with wandering subspaces.  In  Section 3, we studied a characterization of invariant subspaces for a pure isometric covariant representation of the product system over $\mathbb{N}_0^{2}$. Sections 4 and 5 explore the theory of joint defect operator, Fringe operators, and Fredholm index for an isometric covariant representation of product system over $\mathbb{N}_0^2$. Also, we introduce the notion of the congruence relation to classify the isometric covariant representations, with finite defect, of the product system over $\mathbb{N}_0^2$.  
\subsection{ $\mathbf{Preliminaries\:\: and\:\: Notations}$}
Here we shall recall some basic definitions and results from \cite{P73,L95,MS98}. Let $\mathcal{A}$ be a $C^*$-algebra and $E$ be a Hilbert $C^*$-module over $\mathcal{A}.$ We denote $\mathcal L(E)$ to be the collection of all adjointable operators on $E,$ then $\mathcal L(E)$ is a $C^*$-algebra.
The Hilbert module $E$ is said to be {\it
	$C^*$-correspondence  over $\mathcal{A}$ (or Hilbert $\mathcal{A}$-$\mathcal{A}$-module}) if $E$ has a left $\mathcal
A$-module structure induced by  a non-zero $*$-homomorphism
$\phi:\mathcal A\to \mathcal L(E),$ that is,
\[
a\eta=\phi(a)\eta \quad \quad (a\in\mathcal A, \eta\in E).
\]
In this paper, we assume that each $*$-homomorphism is nondegenerate, that is, the closed linear span of $\phi(\mathcal A)E$ equals $E.$ Recall that the Hilbert module  $E$ comes equipped with the operator space structure that inherits as a subspace of the linking algebra \cite{L95}. If $F$ is an another 	$C^*$-correspondence  over $\mathcal{A},$ we denote the interior tensor product of $F$ and $E$ by $F \otimes_{\phi} E$ (cf. \cite{L95}) that satisfies $$\langle \eta_1 \otimes \xi_1, \eta_2 \otimes \xi_2 \rangle= \langle  \xi_1, \phi(\langle \eta_1, \eta_2\rangle )\xi_2\rangle, \: \eta_{1}, \eta_2 \in F, \xi_1, \xi_2 \in E.$$ Unless it is necessary, we simply write $F \otimes E $ instead of $F \otimes_{\phi} E.$

\begin{definition}
	Suppose that $E$ is a $C^*$-correspondence over $\mathcal{A}$ and $\mathcal H$ is a Hilbert space. Let $V:
	E\to B(\mathcal H)$  be a linear map and $\sigma:\mathcal A\to B(\mathcal H)$  be a representation. The pair $(\sigma, V)$ is
	called a {\rm  covariant representation} (cf. \cite{MS98}) of $E$ on $\mathcal H$  if 
	\[
	V(b\eta a)=\sigma(b)V(\eta)\sigma(a) \quad \quad (
	a,b\in\mathcal A, \eta\in E).
	\] 
	Then $(\sigma, V)$ is {\rm completely bounded covariant representation} (simply say, {\rm c.b.c. representation})  if $V$ is completely bounded.
	Further, $(\sigma, V)$ is called {\rm isometric} if $V(\eta)^*V(\xi)=\sigma(\langle \eta,\xi \rangle)$ for all $\xi,\eta\in E.$
\end{definition}
The following lemma is due to Muhly and Solel which is useful to classify the covariant representation of a $C^*$-correspondence.
\begin{lemma}\cite[Lemma 3.5]{MS98} \label{MSC}
	The map $(\sigma, V)\mapsto \wV$ provides a one-to-one correspondence between the set of all c.b.c. representations $(\sigma, V)$
	and the set of all bounded  linear maps $\widetilde{V}:E\otimes_{\sigma} \mathcal H\to \mathcal
	H$ defined by
	\[
	\widetilde{V}(\eta\otimes h)=V(\eta)h \quad \quad (
	h\in\mathcal H,\eta\in E),
	\] such that $\sigma(a)\widetilde{V}=\widetilde{V}(\phi(a)\otimes I_{\mathcal
		H})$, $a\in\mathcal A$. Moreover, $\wV$ is {\rm isometry} if and only if $(\sigma, V)$ is isometric.
\end{lemma}
The c.b.c. representation $(\sigma, V)$ is called fully co-isometric if $\wV$ is co-isometry, i.e., $\wV\wV^*=I_{\mathcal{H}}.$

Suppose that $E$ is a $C^*$-correspondence over $\mathcal{A}$. Then, for  each $n\in \mathbb{N}_0$,
$E^{\otimes n} =E\otimes_{\phi} \cdots \otimes_{\phi}E$ ($n$-times) (here $E^{\otimes 0} =\mathcal{A}$)  is a $C^*$-correspondence over $\mathcal{A}$ in a natural way, with the left module action  of $\mathcal{A}$ on $E^{\otimes n}$ defined as  $$\phi_n(a)(\xi_1 \otimes \cdots \otimes \xi_n)=a\xi_1\otimes \cdots \otimes\xi_n, \: \:\:\: \xi_i \in E,1\le i\le n.$$
For  $n\in \mathbb{N},$ define $\wV_n : E^{\ot n}\ot \mathcal{H} \to \mathcal{H}$ by 
$$\wV_n (\xi_1 \ot \cdots \ot \xi_n \ot h) = V (\xi_1) \cdots V(\xi_n) h \quad \quad \quad (\xi_i \in E, h \in \mathcal H).$$
The \emph{Fock space} of   $E$ (cf. \cite{F02}), $\mathcal{F}(E)= \bigoplus_{n \geq 0}E^{\otimes n},$ is a $C^*$-correspondence over $\mathcal{A},$ where the left module action  of $\mathcal{A}$ on $\mathcal{F}(E)$ is defined  by $$\phi_{\infty}(a)\left(\oplus_{n \geq 0}\eta_n\right)=\oplus_{n \geq 0}\phi_n(a)\eta_n , \:\: \eta_n \in E^{\otimes n}.$$
For   $\xi \in E,$  the \emph{creation operator} $V_{\xi}$ determined by $\xi$ on $\mathcal{F}(E)$ is defined   by $$V_{\xi}(\eta)=\xi \otimes \eta, \:\: \eta \in E^{\otimes n}, n\ge 0.$$
Note that $\|V_{\xi}\|=\|\xi\|$ for all $\xi \in E.$ 
Suppose that  $\pi $ is a representation of $\mathcal{A}$ on a Hilbert space $\mathcal{K}$. An isometric covariant representation $(\rho, S)$ of $E$ on a Hilbert space  $\mathcal{F}(E)\otimes_{\pi}\mathcal{K}$ defined by
\begin{align*}
	S(\xi)=V_{\xi}\otimes I_{\mathcal{K}} \quad and \quad  \rho(a)=\phi_{\infty}(a) \otimes I_{\mathcal{K}} ,\:\: \xi \in E,a \in \mathcal{A}
\end{align*} is called an {\it induced representation} (cf. \cite{RM74}) induced  by $\pi$. Suppose there exists a unitary operator $U: \mathcal{H} \rightarrow \mathcal{F}(E)\otimes_{\pi} \mathcal{K}$  such that  $U\sigma(a)= \rho(a)U$ and $UV(\xi)=S(\xi)U, a \in \mathcal{A}$ and $\xi \in E,$ that is, $(\sigma, V)$ is isomorphic to $(\rho, S),$ then in this case also we say that $(\sigma, V)$ is an {\it induced representation}.

\begin{definition}
	\begin{enumerate}
		\item Let $(\sigma , V)$  be a c.b.c. representation of  ${E}$ on  $\mathcal{H}.$ A non-zero closed subspace $\mathcal{K}\subseteq \mathcal{H}$ is said to be $(\sigma , V)$-{\rm invariant} $(resp. (\sigma , V)$-{\rm reducing}) if it is $\sigma$-invariant and (resp. both $\mathcal{K},\mathcal{K}^{\perp}$) is invariant by each operator  $V(\xi), \xi \in E.$ The restriction gives a new c.b.c. representation $(\sigma , V)|_{\mathcal{K}}$ of $E$ on $\mathcal{K}.$
		
		\item A closed subspace $\mathcal{W}\subseteq \mathcal{H}$ is a called {\rm wandering subspace} for $(\sigma , V)$ if $\mathcal{W}$ is orthogonal to
		$\wV_n(E^{\ot n}\ot \mathcal{W})$ for all $n\in \mathbb{N}$. We say that the wandering subspace $\mathcal{W}$ is {\rm generating} for $(\sigma , V)$ 
		if  $$\mathcal{H}= \bigvee_{n\in \mathbb{N}_0}\wV_n(E^{\ot n}\ot \mathcal{W}).$$ 
	\end{enumerate}
	
\end{definition}

Next, we recall the Wold-von Neumann decomposition given in \cite[Theorem 2.9]{MS99}. We will denote $I$ for $I_{\mathcal{H}}.$
\begin{theorem}[{\bf Muhly-Solel}]\label{L4}
	Suppose that $(\sigma , V)$ is an isometric covariant representation of $E$  on $\mathcal{H}$. Then  $(\sigma, V)$ uniquely decomposes into a direct sum $(\sigma_s, V_s)\bigoplus(\sigma_u, V_u)$ on $\mathcal{H}=\mathcal{H}_s \bigoplus \mathcal{H}_u$  such that $(\sigma_u, V_u)=(\sigma, V)|_{\mathcal{H}_u}$ is a fully co-isometric covariant representation and $(\sigma_s, V_s)= (\sigma, V)|_{\mathcal{H}_s}$ is an induced representation. Moreover,
	$$\mathcal{H}_s=\bigoplus_{n\geq 0} \wV_n(E^{\ot n}\ot \mathcal{W})\:\: \:\mbox{and}\:\:\: \mathcal{H}_u= \bigcap_{n \geq 1}\wV_n(E^{\otimes  n }\otimes  \mathcal{H}),$$ where $\mathcal{W}$ is a {\it wandering subspace} for $(\sigma, V).$ 
	
\end{theorem}

Suppose that $(\sigma , V)$ is an isometric covariant representation of $E$  on  $\mathcal{H}$. Then by Theorem \ref{L4}, there exists a unitary $\Pi_V: \mathcal{H} $ $(=\mathcal{H}_s{\oplus}\mathcal{H}_u)\to (\mathcal{F}(E)\otimes\mathcal{W})\oplus\mathcal{H}_u$ such that
\begin{align*}\Pi_V
	\begin{bmatrix}
		V_s(\xi) & 0 \\
		0 & V_u(\xi)
	\end{bmatrix}	= \begin{bmatrix}
		S(\xi) & 0\\
		0 & V_u(\xi) \\
	\end{bmatrix}\Pi_V
	\quad and \quad 
	\Pi_V
	\begin{bmatrix}
		\sigma_s(a) & 0 \\
		0 & \sigma_u(a)
	\end{bmatrix}	= \begin{bmatrix}
		\rho(a) & 0 \\
		0 & \sigma_u(a)
	\end{bmatrix}\Pi_V.
\end{align*}
In fact, $ \Pi_V(\widetilde{V}_{n}(\eta_n\otimes h)\oplus f) = \widetilde{S}_{n}(\eta_n\otimes h)\oplus f, \quad \eta_n \in E^{\otimes n},h \in \mathcal{W},f \in \mathcal{H}_u.$
Moreover, if  $(\sigma , V)$ is pure  (i.e., SOT-$\lim_{n \rightarrow \infty}\wV_{n}\wV^*_{n}=0 $), then $\mathcal{H}_u$ = $\{0\}$  and we have
$$\Pi_V\sigma(a) = \rho(a)\Pi_V \quad , \quad  \Pi_V V(\xi)=S(\xi)\Pi_V.$$

Therefore, an isometric covariant representation $(\sigma,V)$ of $E$ on ${\mathcal H}$ is pure if and only
if it is isomorphic to the induced representation $(\rho, S)$  of $E$ on  $\mathcal{F}(E)\otimes\mathcal{W}$ for some Hilbert space $\mathcal{W}.$
We say that $\Pi_{V}$ is the {\bf Wold-von Neumann decomposition} of the pure isometric covariant representation $(\sigma, V)$ with the  wandering subspace $\mathcal{W}$.

The aim of this paper is to examine the following Berger-Coburn-Lebow representation (see \cite{BDF06,BCL78,MSS18}) for the pure isometric covariant representation of the product system over $\mathbb{N}_0^2$: 

\begin{theorem}[{\bf Berger-Coburn-Lebow}] Suppose that
	$(V_1, V_2)$ is a pair of commuting isometries on $\mathcal{H}$ and $V = V_1 V_2$ is pure. Then there exist a Hilbert space $\mathcal{K},$ an orthogonal projection $P,$ and a unitary map $U$ on $\mathcal{K}$ such that
	\[
	{\Phi_1}(z)= U^* ( z P^{\perp} + P) \quad \text{and} \quad
	{\Phi_2}(z) = (zP + P^{\perp} ) U \quad \quad (z \in \mathbb{D}),
	\]
	are commuting  inner functions   in $H^\infty_{B(\mathcal{K})}(\mathbb{D})$  and  the triple $(M_{\Phi_1}, M_{\Phi_2}, M_z)$ is unitarily equivalent to $(V_1, V_2, V),$ where $H^\infty_{B(\mathcal{K})}(\mathbb{D})$ is  the set of all $B(\mathcal{K})$- valued bounded analytic functions on $\mathbb{D}.$
\end{theorem}

\section{BCL-representation for a pure isometric covariant representation of a product system over $\mathbb{N}_0^2$}

The classification program for commuting $n$-isometries was proposed and analyzed first by Berger, Coburn, and Lebow \cite{BCL78}. Later on,  Bercovici, Douglas, and Foias \cite{BDF06} gave a complete classification for the commuting $n$-isometries under the assumption that the joint wandering subspace has a finite dimension. Popovici \cite{P04} obtained a Wold-type decomposition for a pair of commuting isometries on $\mathcal{H}$. Skalski and Zacharias \cite{SZ08} discussed the Popovici-Wold-type decomposition for an isometric covariant representation of product system $\mathbb{E}$ over $\mathbb{N}_0^n.$ In \cite{MSS18}, an explicit version of the BCL-representation in terms of multipliers with exact coefficients is given. In this section, we will study BCL-representation for pure isometric covariant representation of the product system over $\mathbb{N}_0^2$ and describe a complete set of (joint) unitary invariants based on \cite{MSS18}.

In this paper, we use $\mathbb N_0 =\{0\}\cup \mathbb N.$  A \emph{product system} $\be$ over $\mathbb N_0^2$  is a family of $C^*$-correspondences  $\{E_1,E_2\}$ with unitary isomorphism $t_{2,1}: E_2 \ot E_1 \to E_1 \ot E_2.$   Now, define  $t_{i,i} = \id_{E_i \ot E_i}, i=1,2 $ and $ t_{1,2} = t_{2,1}^{-1}.$  Therefore,   for all
${\bf{n}}=(n_1, n_2) \in \mathbb N^2_0$ the $C^*$- correspondence $\be({\bf{n}})$  is identified with  $E_1^{\ot^{ n_1}} \ot E_2^{\ot^{n_2}}$ (for more details see \cite{F02, S06, S08}).

\begin{definition}
	Suppose that $\be$ is a product system over $\mathbb N^2_0$ and $\mathcal{H}$ is a Hilbert space.  Let $V^{(i)}:
	E_i\to B(\mathcal H)$, $1 \leq i \leq 2$ be the linear maps and $\sigma:\mathcal A\to B(\mathcal H)$  be a representation. The triple $(\sigma, V^{(1)}, V^{(2)})$ is called {\rm completely bounded covariant representation} (simply say, {\rm c.b.c. representation}) of $\be$ on a Hilbert space $\mathcal{H},$ if each tuple $(\sigma, V^{(i)})$ is a c.b.c. representation of $E_i$ on $\mathcal{H}$ and satisfy the commutative relation
	\begin{equation} \label{rep} \wV^{(i)} (I_{E_i} \ot \wV^{(j)}) = \wV^{(j)} (I_{E_j} \ot \wV^{(i)}) (t_{i,j} \ot I_{\mathcal H}) \quad	\mbox{for} \quad 1 \le i,j\le 2.
	\end{equation}
	Moreover, the covariant  representation  $(\sigma, V^{(1)}, V^{(2)})$ is said to be {\rm isometric (resp. fully co-isometric)} if each pair $(\sigma, V^{(i)})$ is isometric {(resp. fully co-isometric)}.
\end{definition}

Let  $n \in \mathbb{N}$ and $1 \leq  i \leq 2,$  define $\wV^{(i)}_n: E_i^{\ot n} \ot_{\sigma} \mathcal H\to \mathcal H$ by the formula
$$\wV^{(i)}_n (\xi_1 \ot \cdots \ot \xi_n \ot h) = V^{(i)} (\xi_1) \cdots V^{(i)}(\xi_n) h, \quad  \quad (\xi_1, \dots , \xi_n \in E_{i}, h \in \mathcal H ).$$

For ${\bf{n}}=(n_1, n_2) \in \mathbb{N}_0^2,$ we use notation $\wV_{\bf{n}},$ where $\wV_{\bf{n}}: \mathbb{E}(\bf{n}) \ot \mathcal{H} \rightarrow \mathcal{H}$  is given by 
$$\wV_{\bf{n}}=\wV^{(1)}_{n_1}(I_{E_1^{\ot n_1}} \ot \wV^{(2)}_{n_2}).$$
Using Lemma \ref{MSC}, let us define the bi-module map $V_{\bf{n}}: \mathbb{E}({\bf{n}}) \rightarrow B(\mathcal{H})$ by $$V_{\bf{n}}(\xi)h=\wV_{\bf{n}}(\xi \ot h), \:\:\: h \in \mathcal{H},\xi \in \mathbb{E}({\bf{n}}), {\bf{n}} \in \mathbb{N}_0^2.$$

Define the \emph{Fock space} $\mathcal{F}(\mathbb{E})$ of $\mathbb{E}$ over $\mathbb{N}^2_0$ by
$$\mathcal{F}(\mathbb{E})=\bigoplus_{{\bf n} \in \mathbb{N}^2_0}\mathbb{E}({\bf n}).$$

Then $\mathcal{F}(\mathbb{E})$ is a $C^*$-correspondence over  $\mathcal{A},$ with the left module action $\phi_{\infty}$ given by $\phi_{\infty}(a)(\oplus_{{\bf{n}} \in \mathbb{N}_0^2} \xi_{\bf{n}}) = \oplus_{{\bf{n}} \in \mathbb{N}_0^2} \phi_{\bf{n}}(a)\xi_{\bf{n}}, \: a\in \mathcal{A}$ and $\xi_{\bf{n}} \in \mathbb{E}(\bf{n})$ and $\phi_{\bf{n}}$ is the natural left module action of $\mathcal{A}$ on $\mathbb{E}(\bf{n}).$

Let $\mathcal{K}$ be a Hilbert space and $\pi$  be a  representation of $\mathcal{A}$ on  $\mathcal{K}.$  For $1 \leq i \leq 2,$ define an isometric covariant representation $(\rho, S^{(i)})$ of $E_i$ on $\mathcal{F}(\mathbb{E}) \otimes_{\pi} \mathcal{K}$ by \begin{align*}
	\rho(a)=\phi_{\infty}(a) \otimes I_{\mathcal{K}} \quad and \quad
	S^{(i)}(\xi_i)=V_{{\xi}_i} \otimes I_{\mathcal{K}},\:\:\: a \in \mathcal{A}\:, \xi_i \in E_i,
\end{align*}
where $V_{\xi_i}$ is the creation operator determined by $\xi_i$ on $\mathcal{F}(\mathbb{E}).$ The above covariant representation $(\rho, S^{(1)}, S^{(2)})$ is called an {\it induced representation} of $\mathbb{E}$ induced by $\pi.$ Any covariant representation of $\mathbb{E}$ which is isomorphic to the induced representation  $(\rho, S^{(1)}, S^{(2)})$ is also called an {induced representation}.
More generally, we say two such a covariant representations $(\sigma , V^{(1)},V^{(2)})$ and $(\sigma' , V^{(1)'},V^{(2)'})$ of $\mathbb{E}$ on the Hilbert spaces $\mathcal{H}$ and $\mathcal{H}',$ respectively, are {\it isomorphic} (cf. \cite{SZ08})  if there exists a unitary $U:\mathcal H \to
\mathcal H'$ such that   $U\sigma(a)=\sigma'(a)U$ and $V^{(i)'}(\xi_i)U = U V^{(i)} (\xi_i),  a\in \mathcal{A},\xi_i \in E_i, 1 \leq i\leq 2.$

Suppose that $\mathbb{E}$ is a product system over $\mathbb{N}_0^2$. Let $\Theta: E_2 \longrightarrow B(\mathcal{K}, \mathcal{F}(E_1) \otimes_{\pi} \mathcal{K})$ be a completely bounded {\it bi-module map}, that is, $\Theta$ is   completely bounded  and $\Theta(a \xi b)=\rho(a)\Theta(\xi)\rho(b),$ where $\xi \in E_2, a,b \in \mathcal{A}.$ Define a bounded linear map $\widetilde{\Theta}: E_2 \otimes \mathcal{K} \longrightarrow \mathcal{F}(E_1) \otimes_{\pi} \mathcal{K}$  by $\widetilde{\Theta}(\xi \otimes h)=\Theta(\xi)h$ for all $\xi\in E_2,h\in \mathcal{K},$ and it satisfies $\widetilde{\Theta}(\phi_2(a) \ot I_{ \mathcal{K}})=\rho(a)\widetilde{\Theta},$ where $\phi_2$ is the left action of $\mathcal{A}$ on $E_2$ and $a \in \mathcal{A}.$ 
We define a corresponding completely bounded bi-module map $M_{\Theta} :E_2  \longrightarrow B(\mathcal{F}(E_1) \otimes_{\pi} \mathcal{K})$ by
\begin{align*}
	M_{\Theta}(\xi) (S_n^{(1)}(\xi_n) h)=\widetilde{S}_n^{(1)}(I_{E_1^{\otimes n}} \otimes \widetilde{\Theta})(t_{2,1}^{(1,n)} \otimes I_{\mathcal{K}})(\xi \otimes \xi_n \otimes h),
\end{align*}
where  $\xi \in E_2, \xi_n \in E_1^{\otimes n}, h\in \mathcal{K}, n \in \mathbb{N}_0$ and $t_{2,1}^{(1,n)}: E_2 \ot E_1^{\ot n}  \rightarrow E_1^{\ot n} \ot E_2 $ is an isomorphism which is a composition of the isomorphisms $\{t_{i,j}\: :\: 1 \leq i, j \leq 2\}.$
Clearly  $M_{\Theta}(\xi)|_{\mathcal{K}}=\Theta(\xi)$ for each $\xi \in E_2,$  $(\rho,{M}_{\Theta})$ is a c.b.c. representation of $E_2$ on $\mathcal{F}(E_1) \otimes_{\pi} \mathcal{K},$ and it satisfies
\begin{align}\label{I_1}
	M_{\Theta}(\xi)\left(\bigoplus_{n \in \mathbb{N}_0}\xi_n \otimes h_n\right)=\sum_{n \in \mathbb{N}_0}\widetilde{S}_n^{(1)}(I_{E_1^{\otimes n}} \otimes \widetilde{\Theta})(t_{2,1}^{(1,n)} \otimes I_{\mathcal{K}})(\xi \otimes \xi_n \otimes h_n),
\end{align} 
where $\xi \in E_2, \xi_n \in E_1^{\otimes n}, h_n \in \mathcal{K}.$  Let  $(\rho, S)$ be the induced representation of $E_1$ induced by $\pi,$ then it is easy to see that
$\widetilde{M}_{\Theta} (I_{E_2} \otimes \widetilde{S})=\widetilde{S} (I_{{E}_1} \otimes \widetilde{M}_{\Theta})(t_{2,1} \otimes I_{\mathcal{F}(E_1) \otimes \mathcal{K}}).$ That is, $(\rho, S, M_{\Theta})$ is a c.b.c. representation of $\mathbb{E}$ on $\mathcal{F}(E_1) \otimes _{\pi} \mathcal{K} .$
Also, observe  that $\widetilde{\Theta}$  is an isometry if and only if   $(\rho, M_{\Theta})$ is an isometric covariant representation. In this section,  we only focus on $\Theta$ such that $(\rho,M_{\Theta})$ is a c.b.c. representation.

\begin{lemma}\label{L_1}
	Suppose that $\mathbb{E}$ is a product system over $\mathbb{N}_0^2$. Let  $(\rho ,S)$ be the induced representation of $E_1$  induced by $\pi$ and let  $(\rho ,V)$ be a c.b.c. representation of $E_2$ on $\mathcal{F}(E_1) \otimes _{\pi} \mathcal{K}.$  Then $(\rho, S, V)$ is a c.b.c. representation of $\mathbb{E}$ on $\mathcal{F}(E_1) \otimes _{\pi} \mathcal{K}$ if and only if there exists a completely bounded bi-module map  $\Theta: E_2 \longrightarrow  B(\mathcal{K}, \mathcal{F}(E_1) \otimes _{\pi} \mathcal{K}) $ such that $V=M_{\Theta}.$
\end{lemma}
\begin{proof}
	Suppose that  $(\rho, S, V)$  is a c.b.c. representation of $\mathbb{E}$ on $\mathcal{F}(E_1) \otimes \mathcal{K}.$ Then, for $n \in \mathbb{N}$
	\begin{align}\label{I_2}
		\widetilde{V} (I_{E_2} \otimes \widetilde{S}_{n})=\widetilde{S}_{n} (I_{{E}_1^{\otimes n}} \otimes \widetilde{V})(t_{2,1}^{(1,n)} \otimes I_{\mathcal{F}(E_1) \otimes \mathcal{K}}).
	\end{align}
	Define $\Theta: E_2 \longrightarrow {B}(\mathcal{K}, \mathcal{F}(E_1)\otimes_{\pi} \mathcal{K}) $ by $\Theta(\xi)=V(\xi)|_{\mathcal{K}}, \:\xi \in E_2.$ By Equation (\ref{I_1}), for each $\xi \in E_2$, $V(\xi)$ is uniquely determined by $\Theta(\xi).$
	Indeed, since $ \mathcal{F}(E_1) \otimes _{\pi} \mathcal{K}= \bigoplus_{n \in \mathbb{N}_0}\wt{S}_n(E_1^{\otimes n }\otimes \mathcal{K})$ and  for each $\xi \in E_2, \xi_n \in E_1^{\otimes n}, h \in \mathcal{K},$ we have
	\begin{align*} 
		V(\xi)S_n(\xi_n)h&=\widetilde{V}(\xi \otimes S_n(\xi_n)h)=\widetilde{S}_n(I_{E_1^{\otimes n}} \otimes \widetilde{V})(t_{2,1}^{(1,n)} \otimes I_{\mathcal{K}})(\xi \otimes \xi_n \otimes h)
		\\&=\widetilde{S}_n(I_{E_1^{\otimes n}} \otimes \widetilde{\Theta})(t_{2,1}^{(1,n)} \otimes I_{\mathcal{K}})(\xi \otimes \xi_n \otimes h)
		=M_{\Theta}(\xi)  S_n(\xi_n) h.
	\end{align*}
	It follows that $V$ is uniquely determined by $\Theta$ and hence   $V=M_{{\Theta}}.$ The converse part follows from Equation (\ref{I_1}).
\end{proof}

\begin{theorem}\label{L6}
	Suppose that $\mathbb{E}$ is a product system over $\mathbb{N}_0^2$. Let  $(\sigma , V^{(1)})$  be  a  pure isometric covariant representation of $E_1$ on  $\mathcal{H}$ and  $\Pi_{V^{(1)}}$ be the Wold-von Neumann decomposition  of  $(\sigma, V^{(1)})$ with the wandering subspace $\mathcal{W}_1$. Let $(\sigma, V^{(2)})$ be  an  isometric covariant representation of $E_2$ on  $\mathcal{H}.$ Define an isometric covariant representation $(\rho, T)$ of $E_2$ on  $\mathcal{F}(E_1) \otimes_{\pi} \mathcal{W}_1$ by  $$ T(\xi)=\Pi_{V^{(1)}} V^{(2)}(\xi) \Pi_{V^{(1)}}^*,$$ where $ \pi= \sigma|_{\mathcal{W}_1}, \rho$ is defined as above and $ \xi \in E_2.$
	Then  $(\sigma , V^{(1)},V^{(2)})$  is an isometric covariant representation of $\mathbb{E}$ on  $\mathcal{H}$   if and only if there exists  an isometric bi-module map $\Theta : E_2 \rightarrow {B}(\mathcal{W}_1, \mathcal{F}(E_1)\otimes_{\pi} \mathcal{W}_1)$  (i.e., $\widetilde{\Theta}$ is an isometry) such that $T=M_{\Theta}.$ Moreover,  $\Theta(\xi)=\sum_{n \in \mathbb{N}_0}\widetilde{S}_{n}(I_{E_1^{\otimes n}} \otimes P_{\mathcal{W}_1})\widetilde{V}_{n}^{(1)^*}V^{(2)}(\xi), ~\xi \in E_2$ and  $P_{\mathcal{W}_1}$ is an orthogonal projection of $\mathcal{H}$ onto $\mathcal{W}_1.$
\end{theorem}
\begin{proof}
	Suppose that  $(\sigma , V^{(1)},V^{(2)})$  is an isometric covariant representation of $\mathbb{E}$  on  $\mathcal{H}.$  Since $\Pi_{V^{(1)}}\widetilde{V}^{(1)}=\widetilde{S}(I_{E_1} \otimes \Pi_{V^{(1)}}),$ we have
	\begin{align*}
		\widetilde{T}(I_{E_2} \otimes \widetilde{S})&=\Pi_{V^{(1)}}\widetilde{V}^{(2)}(I_{E_2} \otimes \Pi^*_{V^{(1)}}\widetilde{S})
		=\Pi_{V^{(1)}}\widetilde{V}^{(2)}(I_{E_2} \otimes \widetilde{V}^{(1)}(I_{E_1} \otimes \Pi^*_{V^{(1)}}))
		\\&=\Pi_{V^{(1)}}\widetilde{V}^{(1)}(I_{E_1} \otimes \widetilde{V}^{(2)})(t_{2,1} \otimes I_{\mathcal{H}})( I_{E_2 \otimes E_1}\otimes \Pi^*_{V^{(1)}})
		\\&=\widetilde{S}(I_{E_1} \otimes \Pi_{V^{(1)}})(I_{E_1} \otimes \widetilde{V}^{(2)})(t_{2,1} \otimes \Pi^*_{V^{(1)}})
		=\widetilde{S}(I_{E_1} \otimes \widetilde{T})(t_{2,1} \otimes I_{\mathcal{F}(E_1) \otimes \mathcal{W}_1}).
	\end{align*}
	Therefore $(\rho, S, T)$ is an isometric covariant representation of $\mathbb{E}$ on $\mathcal{F}(E_1) \otimes \mathcal{W}_1$. By Lemma \ref{L_1}, there exists an isometric bi-module map $\Theta: E_2 \rightarrow B(\mathcal{W}_1, \mathcal{F}(E_1) \otimes \mathcal{W}_1)$ such that $T=M_{\Theta}.$  Since $(\sigma, V^{(1)})$ is pure, we get
	\begin{align*}
		\sum_{n \in \mathbb{N}_0} \widetilde{V}_{n}^{(1)}(I_{E_1^{\otimes n}} \otimes P_{\mathcal{W}_1})\widetilde{V}_{n}^{(1)^*} &=\sum_{n \in \mathbb{N}_0} \widetilde{V}_{n}^{(1)}(I_{E_1^{\otimes n}} \otimes (I_{\mathcal{H}}-\widetilde{V}^{(1)} \widetilde{V}^{(1)^*}))\widetilde{V}_{n}^{(1)^*} =I_{\mathcal{H}}.
	\end{align*}
	
	If $w \in \mathcal{W}_1 $, then ${\Pi^*_{V^{(1)}}}(w)=w$ and it follows from the previous equality that    $$V^{(2)}(\xi)w=\sum_{n  \in \mathbb{N}_0} \widetilde{V}_{n}^{(1)}(I_{E_1^{\otimes n}} \otimes P_{\mathcal{W}_1})\widetilde{V}_{n}^{(1)^*}V^{(2)}(\xi)w, \:\:  \xi \in  E_2.$$ Therefore\begin{align*}
		\Pi_{V^{(1)}}V^{(2)}(\xi)w&=\Pi_{V^{(1)}}\sum_{n \in \mathbb{N}_0} \widetilde{V}_{n}^{(1)}(I_{E_1^{\otimes n}} \otimes P_{\mathcal{W}_1})\widetilde{V}_{n}^{(1)^*}V^{(2)}(\xi)w =\sum_{n \in \mathbb{N}_0}\widetilde{S}_{n}(I_{E_1^{\otimes n}} \otimes P_{\mathcal{W}_1})\widetilde{V}_{n}^{(1)^*}V^{(2)}(\xi)w.
	\end{align*} Hence  for  $\xi \in E_2,$ we have $\Theta(\xi)=\sum_{n \in \mathbb{N}_0}\widetilde{S}_{n}(I_{E_1^{\otimes n}} \otimes P_{\mathcal{W}_1})\widetilde{V}_{n}^{(1)^*}V^{(2)}(\xi).$

	Conversely, suppose that $T=M_{\Theta},$ then by  Lemma \ref{L_1},  $(\rho, S, T)$ is a c.b.c. representation of $\mathbb{E}$ on $\mathcal{F}(E_1)\otimes \mathcal{W}_1.$ Further, we have 
	\begin{align*}
		\wV^{(1)}(I_{E_1}\ot \wV^{(2)})&=\Pi_{V^{(1)}}^*\widetilde{S}(I_{E_1}\ot \widetilde{T}(I_{E_2}\ot  \Pi_{V^{(1)}}))
		\\&=\Pi_{V^{(1)}}^* \widetilde{T}(I_{E_2}\ot \widetilde{S})(t_{1,2}\ot I_{\mathcal{F}(E_1) \otimes \mathcal{W}_1})(I_{E_1\ot E_2}\ot  \Pi_{V^{(1)}})
		\\&=\Pi_{V^{(1)}}^* \widetilde{T}(I_{E_2}\ot \Pi_{V^{(1)}}\wV^{(1)}(I_{E_1}\ot \Pi_{V^{(1)}}^*) )(t_{1,2}\ot \Pi_{V^{(1)}})\\&=\wV^{(2)}(I_{E_2}\ot \wV^{(1)})(t_{1,2}\ot I_{\mathcal{H}}).
	\end{align*} Therefore $(\sigma , V^{(1)},V^{(2)})$ is an isometric covariant  representation of $\mathbb{E}$ on $\mathcal{H}.$
\end{proof}

Suppose that  $(\sigma , V^{(1)},V^{(2)})$ is an isometric covariant representation of $\mathbb{E}$ over $\mathbb{N}_0^2$ on $\mathcal{H}.$ Define  $\widetilde{V}=\widetilde{V}^{(1)}(I_{E_1} \otimes \widetilde{V}^{(2)}),$ then by Lemma \ref{MSC} $(\sigma, V)$ is also an isometric covariant represenation of $E:=E_1 \ot E_2.$
If $(\sigma, V)$ is pure, then we say  that $(\sigma , V^{(1)},V^{(2)})$ is pure.
Throughout this paper, we will use the following symbols unless otherwise stated:	
\[ \mathcal{W}=\mathcal{H} \ominus \widetilde{V}(E  \otimes \mathcal{H}) \:\:\: \mbox{and} \:\:\: \mathcal{W}_{i}=\mathcal{H} \ominus \widetilde{V}^{(i)}(E_i \otimes \mathcal{H}), \quad 1 \leq i\leq 2. \]

Note that
\begin{align*}
	I-\widetilde{V}\widetilde{V}^* &=
	I-\widetilde{V}^{(1)}\widetilde{V}^{(1)^*}+\widetilde{V}^{(1)}\widetilde{V}^{(1)^*}-\widetilde{V}^{(1)}(I_{E_1}\otimes\widetilde{V}^{(2)}\widetilde{V}^{(2)^*})\widetilde{V}^{(1)^*}\\&=I-\widetilde{V}^{(1)}\widetilde{V}^{(1)^*}+\widetilde{V}^{(1)}(I_{E_1}\otimes(I_{\mathcal{H}}-\widetilde{V}^{(2)}\widetilde{V}^{(2)^*}))\widetilde{V}^{(1)^*}.
\end{align*}
Similarly, 	$I-\widetilde{V}\widetilde{V}^* =I-\widetilde{V}^{(2)}\widetilde{V}^{(2)^*}+\widetilde{V}^{(2)}(I_{E_2}\otimes(I_{\mathcal{H}}-\widetilde{V}^{(1)}\widetilde{V}^{(1)^*}))\widetilde{V}^{(2)^*}.$ Then
\begin{align}\label{WWW}
	\mathcal{W}=\mathcal{W}_{1} \oplus \widetilde{V}^{(1)}(E_1 \otimes \mathcal{W}_{2})=\widetilde{V}^{(2)}(E_2 \otimes \mathcal{W}_{1}) \oplus \mathcal{W}_{2}.
\end{align}


For $1 \leq i \leq 2,$ we  have
\begin{align*}
	\wV^{(i)}(I_{E_i}\ot \wV) &=\wV^{(i)}(I_{E_i}\ot \wV^{(1)})(I_{E_i}\ot I_{E_1} \otimes \widetilde{V}^{(2)})=\wV^{(1)}(I_{E_1}\ot \wV^{(i)})(t_{i,1}\ot \wV^{(2)})\\&=\wV^{(1)}(I_{E_1}\ot \wV^{(i)}(I_{E_i}\ot \wV^{(2)}))(t_{i,1}\ot I_{E_2} \ot I_{\mathcal{H}})\\&=\wV^{(1)}(I_{E_1}\ot \wV^{(2)}(I_{E_2}\ot \wV^{(i)})(t_{i,2}\ot I_{\mathcal{H}}))(t_{i,1}\ot I_{E_2} \ot I_{\mathcal{H}})\\&=\wV(I_{E}\ot \wV^{(i)})((I_{E_1}\ot t_{i,2})(t_{i,1}\ot I_{E_2})\ot I_{\mathcal{H}})=\wV(I_{E}\ot \wV^{(i)})(t_{i,\bf{1}} \ot I_{\mathcal{H}}),
\end{align*} where $t_{i,\bf{1}}:E_i \ot E \to E\ot E_i$ is an isomorphism. Therefore $(\sigma , V^{(1)},V^{(2)}, V)$ is also an isometric covariant representation of the product system  over $\mathbb{N}_0^3$ determined by the $C^*$-correspondences $\{E_1, E_2, E\}$ on $\mathcal{H}.$ 
Suppose that $(\sigma , V^{(1)},V^{(2)})$ is  pure  and   $\Pi_V$  is  the Wold-von Neumann decomposition of $(\sigma , V)$ with the wandering subspace $\mathcal{W}.$ Since  $(\sigma , V,  V^{(i)}), 1 \leq i \leq 2,$ is an isometric covariant representation and using Theorem \ref{L6},  there exist  isometric bi-module maps $\Theta_i : E_i \to {B}(\mathcal{W}, \mathcal{F}({E})\otimes \mathcal{W})$ such that \begin{align}\label{BBB}
	\Pi_V V^{(i)}(\xi_i)\Pi_V^*=M_{\Theta_i}(\xi_i),\quad \xi_i \in E_i, 1 \leq  i \leq 2.\end{align}
It follows  from the above equation, the triple  $(\rho,M_{{\Theta_1}}, M_{{\Theta_2}})$ is a pure isometric covariant representaion of $\mathbb{E}$ on $ \mathcal{F}({E})\otimes \mathcal{W},$ where $\rho: \mathcal{A} \rightarrow B( \mathcal{F}({E})\otimes \mathcal{W})$ defined by $\rho(a)=\phi_{\infty}(a)\ot I_{ \mathcal{W}}, a\in \mathcal{A}$ and $ \phi_{\infty}$ is a left module action of $\mathcal{A}$ on $ \mathcal{F}({E}).$ The covariant representation $(\rho,M_{{\Theta_1}}, M_{{\Theta_2}})$  is called  the  {\it BCL-representation} for the pure isometric covariant representation $(\sigma , V^{(1)},V^{(2)})$ of the product system $\mathbb{E}$ over $\mathbb{N}_0^2.$

Now we present an explicit construction for the isometric bi-module maps $(\Theta_1, \Theta_2)$ of the BCL-representation $(\rho,M_{{\Theta_1}}, M_{{\Theta_2}}).$ 

\begin{theorem}\label{BCL}
	Let $(\sigma , V^{(1)},V^{(2)})$ be a pure  isometric covariant representation of $\mathbb{E}$ on $\mathcal{H}$ and $(\rho,M_{{\Theta_1}}, M_{{\Theta_2}})$ be the BCL-representation of $(\sigma , V^{(1)},V^{(2)}).$ Then 
	\begin{equation}\label{H7}
		\widetilde{\Theta}_1 = \widetilde{S}(I_{E_1}\ot \wV^{(2)*}|_{\wV^{(2)}(E_2 \ot \mathcal{W}_1)}) \oplus \wV^{(1)}|_{E_1\ot \mathcal{W}_2} 
	\end{equation}
	and
	\begin{equation}
		\widetilde{\Theta}_2 = \widetilde{S}(t_{2,1}\ot I_{\mathcal{W}})(I_{E_2}\ot \wV^{(1)*}|_{\wV^{(1)}(E_1 \ot \mathcal{W}_2)}) \oplus \wV^{(2)}|_{E_2\ot \mathcal{W}_1} ,
	\end{equation} where $(\rho, S)$ is the induced representation of $E$  induced by $\sigma|_{\mathcal{W}}.$
\end{theorem}
\begin{proof}
	Since $\mathcal{W}$ is $\sigma$-invariant subspace, consider the induced representation $(\rho, S)$ of $E$ induced by $\sigma|_{\mathcal{W}}.$   Let  $w \in \mathcal{W}={\wV^{(1)}(E_1 \ot \mathcal{W}_2)} \oplus \mathcal{W}_1,$ then there exist $w_1 \in \mathcal{W}_1$ and $\eta_1 
	\in E_1 \ot \mathcal{W}_2$ such that $w=\wV^{(1)}\eta_1 + w_1,$  therefore $\wV^{(1)*}w= \eta_1.$   It follows that for $\xi_2 \in E_2,$ we have
	\begin{align*}
		\wV^{(2)}(\xi_2 \ot w)&=\wV^{(2)}(I_{E_2}\ot \wV^{(1)})(\xi_2\ot \eta_1) + \wV^{(2)}(\xi_2 \ot w_1)\\&= \wV(t_{2,1}\ot I_{\mathcal{H}})(\xi_2\ot \eta_1) + \wV^{(2)}(\xi_2 \ot w_1) \: \mbox{and}
	\end{align*}
	\begin{align*}
		\wV^{(2)}(\xi_2 \ot w_1)&=\wV^{(2)}(\xi_2 \ot w)-\wV^{(2)}(I_{E_2}\ot \wV^{(1)})(\xi_2\ot \wV^{(1)*}w)\\&= \wV^{(2)}(I_{E_2}\ot (I-\wV^{(1)}\wV^{(1)*}))(\xi_2 \ot w)=\wV^{(2)}(I_{E_2}\ot P_{\mathcal{W}_1})(\xi_2 \ot w),
	\end{align*}
	where $P_{\mathcal{W}_1}$ is the orthogonal projection onto $\mathcal{W}_1.$
	Using Equation (\ref{WWW}),  we get  $\wV^{(2)}(\xi_2 \ot w_1)\in \mathcal{W}$ and hence 
	\begin{align*}
		{\Theta_2}(\xi_2)w&=M_{{\Theta}_2}(\xi_2)w=\Pi_V V^{(2)}(\xi_2)\Pi_V^{*}w=\Pi_V V^{(2)}(\xi_2)w\\&= \Pi_V \wV(t_{2,1}\ot I_{\mathcal{H}})(\xi_2\ot \eta_1)+ \Pi_V \wV^{(2)}(\xi_2 \ot w_1)\\&= \widetilde{S}(I_{E}\ot \Pi_V )(t_{2,1}\ot I_{\mathcal{H}})(\xi_2\ot \eta_1) + \Pi_V \wV^{(2)}(\xi_2 \ot w_1)\\&= \widetilde{S}(t_{2,1}\ot I_{\mathcal{W}})(\xi_2 \ot \eta_1) + \wV^{(2)}(\xi_2 \ot w_1)\\ &=\widetilde{S}(t_{2,1}\ot I_{\mathcal{W}})(\xi_2 \ot \wV^{(1)*}w) + \wV^{(2)}(I_{E_2}\ot P_{\mathcal{W}_1})(\xi_2 \ot w),
	\end{align*}
	the last equality follows from  $\wV^{(1)*}w= \eta_1 .$ Therefore 
	\begin{equation*}
		\widetilde{\Theta}_2 = \widetilde{S}(t_{2,1}\ot I_{\mathcal{W}})(I_{E_2}\ot \wV^{(1)*}|_{\wV^{(1)}(E_1 \ot \mathcal{W}_2)}) \oplus \wV^{(2)}|_{E_2\ot \mathcal{W}_1}.
	\end{equation*}	
	Similarly, we get the relation for $\widetilde{\Theta}_1.$
\end{proof}

The proof of the following theorem follows from  Theorem \ref{BCL}.

\begin{theorem}\label{K2}
	Suppose that $(\sigma , V^{(1)},V^{(2)})$ is a pure isometric covariant representation of $\mathbb{E}$ on $\mathcal{H}.$ Then the BCL-representation $(\rho,M_{{\Theta_1}}, M_{{\Theta_2}})$ of $(\sigma , V^{(1)},V^{(2)})$ is given by $$\widetilde{\Theta}_1 = (P_{\mathcal{W}_1}^{\perp} + \widetilde{S}(I_{E_1\ot E_2}\ot P_{\mathcal{W}_1}))U'
	$$ and $$
	\widetilde{\Theta}_2 = (P_{\mathcal{W}_2}^{\perp} + \widetilde{S}(t_{2,1}\ot P_{\mathcal{W}_2}))U,
	$$ where
	$U'= \begin{bmatrix}
		\wV^{(1)}|_{E_1 \ot \mathcal{W}_2} & 0 \\
		0 & (I_{E_1}\ot \wV^{(2)*}|_{\wV^{(2)}(E_2 \ot \mathcal{W}_1)})
	\end{bmatrix} : E_1 \ot \mathcal{W}_2 \oplus E_1 \ot \wV^{(2)}(E_2 \ot \mathcal{W}_1) \to 	\wV^{(1)}(E_1 \ot \mathcal{W}_2) \oplus E \ot \mathcal{W}_1$ and 
	
	$U= \begin{bmatrix}
		\wV^{(2)}|_{E_2 \ot \mathcal{W}_1} & 0\\
		0 & (I_{E_2}\ot \wV^{(1)*}|_{\wV^{(1)}(E_1 \ot \mathcal{W}_2)})
	\end{bmatrix} : E_2 \ot \mathcal{W}_1 \oplus E_2\ot \wV^{(1)}(E_1 \ot \mathcal{W}_2) \to \wV^{(2)}(E_2 \ot \mathcal{W}_1) \oplus E_2 \ot E_1 \ot \mathcal{W}_2$  are unitary isomorphisms.
	
	
\end{theorem}

\begin{remark}\label{K3}
	Suppose that $(\sigma , V^{(1)},V^{(2)})$ is a pure  isometric covariant representation of $\mathbb{E}$ on $\mathcal{H}.$ Using Equation $(\ref{H7})$,  define a bounded linear map $(I_{E_2},I_{\mathcal{A}})\ot \widetilde{\Theta}_1:E_2\ot E_1 \ot \mathcal{W}_2 \oplus  E_1 \ot \wV^{(2)}(E_2 \ot \mathcal{W}_1) \to E_2 \ot \wV^{(1)}(E_1 \ot \mathcal{W}_2) \oplus E_1\ot E_2 \ot \mathcal{W}_1$ by
	$$(I_{E_2},I_{\mathcal{A}})\ot \widetilde{\Theta}_1= I_{E_2} \ot \wV^{(1)}|_{E_1\ot \mathcal{W}_2} \oplus \widetilde{S}(I_{E_1}\ot \wV^{(2)*}|_{\wV^{(2)}(E_2 \ot \mathcal{W}_1)}),$$ and it satisfies
	$((I_{E_2},I_{\mathcal{A}})\ot \widetilde{\Theta}_1)(\phi_2(a)\ot I_{E_1}\ot I_{\mathcal{W}_2}\oplus \phi_1(a)\ot I_{\wV^{(2)}(E_2 \ot \mathcal{W}_1)})=(\phi_2(a)\ot I_{\wV^{(1)}(E_1 \ot \mathcal{W}_2)}\oplus \phi_1(a)\ot I_{E_2}\ot I_{\mathcal{W}_1})((I_{E_2},I_{\mathcal{A}})\ot \widetilde{\Theta}_1),$ where $\phi_i$ is the left action of $\mathcal{A}$ on $E_i,1 \leq i \leq 2$ and $a \in \mathcal{A}.$	Also,  define a unitary isomorphism $(I_{E_1},I_{\mathcal{A}})\ot U:E_1\ot E_2 \ot \mathcal{W}_1 \oplus  E_2 \ot \wV^{(1)}(E_1 \ot \mathcal{W}_2) \to E_1 \ot \wV^{(2)}(E_2 \ot \mathcal{W}_1) \oplus E_2\ot E_1 \ot \mathcal{W}_2$ by 
	$$(I_{E_1},I_{\mathcal{A}})\ot U=I_{E_1} \ot \wV^{(2)}|_{E_2 \ot \mathcal{W}_1} \oplus  (I_{E_2}\ot \wV^{(1)*}|_{\wV^{(1)}(E_1 \ot \mathcal{W}_2)}) .$$ Then,  we conclude that
	$$(I_{E_2},I_{\mathcal{A}})\ot \widetilde{\Theta}_1=( (I_{E_1},I_{\mathcal{A}})\ot U^*)( I_{E_1}\ot P_{\mathcal{W}_2}^{\perp} \oplus  I_{E_2 \ot E_1}\ot P_{\mathcal{W}_2} ).$$
\end{remark}

The following theorem gives a complete set of (joint) unitary invariants for a pure isometric covariant representation of the product system over $\mathbb{N}_0^2.$ 
\begin{theorem}
	Let  $(\sigma , V^{(1)},V^{(2)})$ and $(\sigma' , V^{(1)'},V^{(2)'})$ be two pure isometric covariant representation of $\mathbb{E}$ on  the Hilbert spaces $\mathcal{H}$ and $\mathcal{H}',$ respectively. Then  $(\sigma , V^{(1)},V^{(2)})$ and $(\sigma' , V^{(1)'},V^{(2)'})$  are isomorphic if and only if $(\sigma|_{\mathcal{W}},\wV^{(1)}|_{E_1\ot \mathcal{W}_2},I_{E_1}\ot \wV^{(2)*}|_{\wV^{(2)}(E_2 \ot \mathcal{W}_1)})$ and $(\sigma'|_{\mathcal{W}'},\wV^{(1)'}|_{E_1\ot \mathcal{W}_2'},I_{E_1}\ot \wV^{(2)'^*}|_{\wV^{(2)'}(E_2 \ot \mathcal{W}_1')})$ are isomorphic, where $\mathcal{W}_i'$ and $\mathcal{W}'$ are the wandering subspaces  for $(\sigma' , V^{(i)'})$ and  $(\sigma' , V'),$ respectively.
\end{theorem}
\begin{proof}
	Consider the induced representations  $(\rho, S)$ and $(\rho', S')$  of $E$ induced by $\sigma|_{\mathcal{W}}$ and $\sigma'|_{\mathcal{W'}},$ respectively. Let $(\rho,M_{{\Theta_1}}, M_{{\Theta_2}})$ and $(\rho',M_{{\Theta_1'}}, M_{{\Theta_2'}})$ be the BCL-representation of $(\sigma , V^{(1)},V^{(2)})$ and  $(\sigma' , V^{(1)'},V^{(2)'}),$  respectively. Suppose $(\sigma|_{\mathcal{W}},\wV^{(1)}|_{E_1\ot \mathcal{W}_2},I_{E_1}\ot \wV^{(2)^*}|_{\wV^{(2)}(E_2 \ot \mathcal{W}_1)})$ and  $(\sigma'|_{\mathcal{W}'},\\\wV^{(1)'}|_{E_1\ot \mathcal{W}_2'} , I_{E_1} \ot \wV^{(2)'^*}|_{\wV^{(2)'}(E_2 \ot \mathcal{W}_1')})$  are isomorphic, this means that, there exists a unitary operator $X: \mathcal{W} \to \mathcal{W}'$ such that   $X\sigma(a)|_{\mathcal{W}}=\sigma'(a)X,$     $(I_{E_2}\ot X){\psi}_1={\psi}_1'X$ and $ X{\psi}_2={\psi}_2'(I_{E_1}\ot X),$ where ${\psi}_1= \wV^{(2)*}|_{\wV^{(2)}(E_2 \ot \mathcal{W}_1)}$, ${\psi}_2=\wV^{(1)}|_{E_1\ot \mathcal{W}_2},$ ${\psi}_1'= \wV^{(2)'^*}|_{\wV^{(2)'}(E_2 \ot \mathcal{W}_1')}$ and  ${\psi}_2'=\wV^{(1)'}|_{E_1\ot \mathcal{W}_2'}.$ Then
	\begin{align*}
		\widetilde{M}_{{\Theta}_1}|_{E_1\ot E^{\ot n}\ot \mathcal{W}}&=\widetilde{S}_n(I_{E^{\ot n}}\ot \widetilde{\Theta}_1)(t_{1,{\bf 1}}^{(1,n)}\ot I_{\mathcal{W}})=\widetilde{S}_n(I_{E^{\ot n}}\ot ((I_{E_1}\ot {\psi}_1) \oplus {\psi}_2))(t_{1,{\bf 1}}^{(1,n)}\ot I_{\mathcal{W}})\\&= (I_{E^{\ot n}}\ot ((I_{E} \ot X^*) \oplus X^*)(I_{E_1}\ot{\psi}_1' \oplus {\psi}_2')(I_{E_1}\ot X))(t_{1,{\bf 1}}^{(1,n)}\ot I_{\mathcal{W}})\\&=(I_{E^{\ot n}}\ot ((I_{E} \ot X^*) \oplus X^*)\widetilde{\Theta}_1'(I_{E_1}\ot X))(t_{1,{\bf 1}}^{(1,n)}\ot I_{\mathcal{W}}),
	\end{align*} where $t_{1,{\bf 1}}^{(1,n)}: E_1 \ot E^{\ot n}  \rightarrow E^{\ot n} \ot E_1 $ is an isomorphism which is a composition of $\{t_{i,j}\: :\: 1 \leq i, j \leq 2\}.$ Define a unitary operator $U:\mathcal{F}({E})\ot \mathcal{W} \to \mathcal{F}({E})\ot \mathcal{W}'$ by $U=I_{\mathcal{F}(E)}\ot X.$ By  using the previous equality, we get
	$$\widetilde{M}_{{\Theta}_1}|_{E_1\ot E^{\ot n}\ot \mathcal{W}}=U^*\widetilde{M}_{{\Theta}_1'}(I_{E_1}\ot U)|_{E_1\ot E^{\ot n}\ot \mathcal{W}},\quad \quad n\in \mathbb{N}_0.$$ Therefore, $\widetilde{M}_{{\Theta}_1}=U^*\widetilde{M}_{{\Theta}_1'}(I_{E_1}\ot U)$ and $U\rho(a)=\rho'(a)U.$
	Since $(I_{E_2}\ot\widetilde{M}_{{\Theta}_1})=\widetilde{M}_{{\Theta}_2}^*\widetilde{S}(t_{2,1}\ot I_{\mathcal{F}(E)\ot \mathcal{W}})$ and $(I_{E_2}\ot\widetilde{M}_{{\Theta}'_1})=\widetilde{M}_{{\Theta}'_2}^*\widetilde{S}'(t_{2,1}\ot I_{\mathcal{F}(E)\ot \mathcal{W}'}),$ it follows that $(\rho,M_{{\Theta_1}}, M_{{\Theta_2}})$ and $(\rho',M_{{\Theta_1'}}, M_{{\Theta_2'}})$ are isomorphic. Hence  $(\sigma , V^{(1)},V^{(2)})$ and $(\sigma' , V^{(1)'},V^{(2)'})$ are isomorphic.
	
	Conversely, suppose that  $(\rho,M_{{\Theta_1}}, M_{{\Theta_2}})$ and $(\rho',M_{{\Theta_1'}}, M_{{\Theta_2'}})$ are isomorphic,  then there exists a unitary operator $U: \mathcal{F}({E})\ot \mathcal{W} \to \mathcal{F}({E})\ot \mathcal{W}'$ such that
	\begin{equation}\label{H1}
		U\rho(a)=\rho'(a)U  \:\: \mbox{and} \:\: U\widetilde{M}_{{\Theta}_i}=\widetilde{M}_{{\Theta}_i'}(I_{E_i}\ot U), \quad a \in \mathcal{A}, 1\leq   i \leq 2.
	\end{equation}  This  yields \begin{align*}
		U\widetilde{S}&=U \widetilde{M}_{{\Theta}_1}(I_{E_1}\ot \widetilde{M}_{{\Theta}_2})= \widetilde{M}_{{\Theta}_1'}(I_{E_1}\ot U \widetilde{M}_{{\Theta}_2}) =  \widetilde{M}_{{\Theta}_1'}(I_{E_1}\ot \widetilde{M}_{{\Theta}_2'}(I_{E_2}\ot U))=\widetilde{S'}(I_E\ot U)
	\end{align*} and by \cite[Proposition 4.3 ]{TV21}, there exists a unitary map $X:\mathcal{W}\to \mathcal{W}'$ such that 
	$U=I_{\mathcal{F}({E})}\ot X.$
	Thus using Equation (\ref{H1}), we get
	\begin{equation}\label{K1}
		(I_{\mathcal{F}({E})}\ot X)\widetilde{M}_{{\Theta}_i}=\widetilde{M}_{{\Theta}_i'}(I_{E_1}\ot  (I_{\mathcal{F}({E})} \ot X)).
	\end{equation}
	
	Since $E_1 \ot \mathcal{W} \subseteq E_1\ot {\mathcal{F}({E})}\ot \mathcal{W},$ using Equation (\ref{K1}) restricted to $E_1 \ot \mathcal{W}$ we get $(I_{\mathcal{F}({E})}\ot X){\widetilde{\Theta}_1}={\widetilde{\Theta}_1'}(I_{E_1}\ot  X),$ and thus by Equation (\ref{H7}) we have $((I_E\ot X) \oplus X)((I_{E_1}\ot {\psi}_1) \oplus {\psi}_2)=((I_{E_1}\ot {\psi}_1') \oplus {\psi}_2')(I_{E_1}\ot  X).$ 
	This shows that  $(\sigma|_{\mathcal{W}},\wV^{(1)}|_{E_1\ot \mathcal{W}_2},I_{E_1}\ot \wV^{(2)*}|_{\wV^{(2)}(E_2 \ot \mathcal{W}_1)})$ and $(\sigma'|_{\mathcal{W'}},\wV^{(1)'}|_{E_1\ot \mathcal{W}_2'},I_{E_1}\ot \wV^{(2)'^*}|_{\wV^{(2)'}(E_2 \ot \mathcal{W}_1')})$  are isomorphic.\end{proof}

The following theorem is an analogue of a classification result due to Bercovici, Douglas and Foias \cite[Theorem 2.1]{BDF06}:
\begin{theorem}\label{BCT}
	Let $\mathbb{E}$ be a product system over $\mathbb{N}_0^2.$ Suppose that  $\pi$ is a representation of $\mathcal{A}$ on a Hilbert space $\mathcal{W}.$ Let $(\rho, S)$  be the induced representation of $E$ induced by $\pi$ and $(\rho, V^{(2)})$ be an isometric covariant representation of $E_2$ on $\mathcal{F}(E) \ot \mathcal{W}.$ The following conditions are equivalent:  
	\begin{enumerate}
		\item There exists an isometric covariant representation $(\rho,V^{(1)} )$ of $E_1$ on  $\mathcal{F}(E) \ot \mathcal W$ such that $(\rho,V^{(1)}, V^{(2)} )$ satisfies (\ref{rep}) and  $\wV^{(1)} (I_{E_1} \ot \wV^{(2)})=\wt S.$
		\item   There exist a closed subspce $\mathcal{L}$ of $\mathcal{W},$ $P:\mathcal{W}\to \mathcal{W}$ is an orthogonal projection onto $\mathcal{L}$ and unitary isomorphisms $U_{\mathcal{L}}:\mathcal{L} \to E_1 \ot P\mathcal{W}$ and  $U:E_2\ot \mathcal{W}\to P^{\perp}{\mathcal{W}}\oplus E_2\ot E_1\ot P\mathcal{W}$ such that   $U|_{E_2 \otimes \mathcal{L}}=I_{E_2} \otimes U_{\mathcal{L}}$ and $V^{(2)}=M_{\Theta_2},$ where $\widetilde{\Theta}_2=(P^{\perp}\oplus \widetilde{S}(t_{2,1}\ot P))U.$
	\end{enumerate} 
\end{theorem} 
\begin{proof}
	$(2) \Rightarrow (1)$ Suppose that $P:\mathcal{W}\to \mathcal{W}$ is an orthogonal projection onto $\mathcal{L}$ and $U:E_2\ot \mathcal{W}\to P^{\perp}{\mathcal{W}}\oplus E_2\ot E_1\ot P\mathcal{W}$ is a unitary isomorphism. Define a bounded linear map $\widetilde{\Theta}_2:E_2\ot \mathcal{W}\to \mathcal{F}(E)\ot \mathcal{W}$ by \begin{align}
		\widetilde{\Theta}_2=(P^{\perp}\oplus \widetilde{S}(t_{2,1}\ot P))U,
	\end{align} then it is easy to verify that $\widetilde{\Theta}_2$ is isometry and $\widetilde{\Theta}_2(\phi_2(a) \ot I_{ \mathcal{W}})=\rho(a)\widetilde{\Theta}_2,$ where $\phi_2$ is the left action of $\mathcal{A}$ on $E_2$ and $a \in \mathcal{A}.$ Define a corresponding bounded linear map $\widetilde{M}_{\Theta_2} :E_2 \ot \mathcal{F}(E) \otimes_{\pi} \mathcal{W} \longrightarrow \mathcal{F}(E) \otimes_{\pi} \mathcal{W}$ by $\widetilde{M}_{\Theta_2}=\sum_{n \in \mathbb{N}_0}\widetilde{S}_n(I_{E^{\otimes n}} \otimes \widetilde{\Theta}_2)(t_{2,\bf{1}}^{(1,n)} \otimes I_{\mathcal{W}}),$ and it satisfies $\widetilde{M}_{\Theta_2}(\phi_2(a) \ot I_{\mathcal{F}(E) \ot \mathcal{W}})=\rho(a)\widetilde{M}_{\Theta_2}$ for $a \in \mathcal{A}.$ Clearly, $\widetilde{M}_{\Theta_2}|_{E_2\ot \mathcal{W}}=\widetilde{\Theta}_2, ker \widetilde{M}^*_{\Theta_2}=P\mathcal{W}$ and  $(\rho,S,M_{{\Theta}_2})$ is an isometric covariant representation of the product system  over $\mathbb{N}_0^2$ determined by $\{E, E_2\}$ on $\mathcal{F}(E)\ot \mathcal{W}.$ 
	Suppose that there exist a  closed  subspace $\mathcal{L} $ of $\mathcal{W}$ and a unitary isomorphism $U_{\mathcal{L}}: \mathcal{L} \to E_1 \otimes P\mathcal{W} $ such that $U|_{E_2 \otimes \mathcal{L}}=I_{E_2} \otimes U_{\mathcal{L}}.$  Since the range of $\widetilde{M}_{\Theta_2}$ equals $P^{\perp}{\mathcal{W}}  \oplus E \ot \mathcal{F}(E)\ot \mathcal{W},$ we
	define an isometric map  $\widetilde{V}^{(1)} :E_1 \ot \mathcal{F}(E) \otimes \mathcal{W} \rightarrow  \mathcal{F}(E) \otimes \mathcal{W}$ by
	$$\widetilde{V}^{(1)}(\xi \ot h) = \begin{cases}   \widetilde{S}(I_{E_1}\ot \widetilde{M}_{{\Theta}_2}^*)(\xi \ot h) &   \text{if }  h \in  P^{\perp}{\mathcal{W}}  \oplus E \ot \mathcal{F}(E)\ot \mathcal{W} \\
		U_{\mathcal{L}}^*(\xi \otimes h)&   \text{if }  h \in P\mathcal{W},
	\end{cases}$$
	for $\xi \in E_1$ and it satisfies 
	$\widetilde{V}^{(1)}(\phi_1(a) \otimes I_{ \mathcal{F}(E)\ot \mathcal{W}})=\rho(a)\widetilde{V}^{(1)},$ where $\phi_1$ is the left action of $\mathcal{A}$ on $E_1$ and $a \in \mathcal{A}.$ Then by  Lemma \ref{MSC},  $(\rho,V^{(1)})$ is an isometric covariant representation of $E_1$ on $\mathcal{F}(E) \otimes \mathcal{W},$  where ${V}^{(1)} :E_1 \to  B(\mathcal{F}(E) \otimes \mathcal{W})$ is defined by   $V^{(1)}(\xi)h=\widetilde{V}^{(1)}(\xi \ot h)$ for all $\xi \in E_1, h\in \mathcal{F}(E) \otimes \mathcal{W}.$
	It is easy to verify that $\widetilde{S}=\wV^{(1)}(I_{E_1}\ot \widetilde{M}_{{\Theta}_2})=\widetilde{M}_{{\Theta}_2}(I_{E_2}\ot \widetilde{V}^{(1)})(t_{1,2}\ot I_{\mathcal{F}(E)\ot \mathcal{W}}),$ that is, $(\rho,V^{(1)},{M}_{{\Theta}_2})$ is an isometric covariant representation of $\mathbb{E}$ on $ \mathcal{F}(E)\ot \mathcal{W}.$ It follows that $(\rho,S,V^{(1)})$ is an isometric covariant representation of the product system  over $\mathbb{N}_0^2$ determined by  $\{E,E_1\}$ on $ \mathcal{F}(E)\ot \mathcal{W},$ then by Lemma \ref{L_1} there exists an isometric bi-module map $\Theta_1: E_1 \longrightarrow  B(\mathcal{W}, \mathcal{F}(E) \otimes  \mathcal{W}) $ such that $V^{(1)}={M}_{{\Theta}_1}.$
	
	$(1) \Rightarrow (2)$ Follows from Theorem \ref{K2}.
\end{proof}

\begin{remark}
	(1) From the above theorem, $\Theta_1$ satisfies   $(I_{E_2},I_{\mathcal{A}})\ot \widetilde{\Theta}_1=( (I_{E_1},I_{\mathcal{A}})\ot U^*)( I_{E_1}\ot P^{\perp} \oplus  I_{E_2 \ot E_1}\ot P), $ where $(I_{E_1},I_{\mathcal{A}})\ot U=I_{E_1} \ot \widetilde{\Theta}_2|_{E_2 \ot \mathcal{W}_1} \oplus  (I_{E_2}\ot \widetilde{M}^*_{\Theta_1}|_{U_{\mathcal{L}}^{*}(E_1 \ot P\mathcal{W})}) $ is a unitary isomorphism and  $\mathcal{W}_1=  ker \widetilde{M}^*_{\Theta_1}.$
	
	(2) Let $(\sigma, V^{(2)},V)$ be an isometric covariant representation of the product system determined by $\{E_2, E\}$ on a Hilbert space $\mathcal{H}$  such that $(\sigma,V)$  is pure. Consider  the induced representation  $(\rho, S)$  of $E$ induced by $\sigma|_{\mathcal{W}},$ where $\mathcal{W}=ker \wt V^*.$ Then by (\ref{BBB}),  $(\rho, M_{\Theta_2}, S)$ is an isometric covariant representation which is isomorphic to $(\sigma, V^{(2)},V).$  Assume statement (2) of Theorem \ref{BCT}, then we get the BCL-representation $(\rho, M_{\Theta_1},M_{\Theta_2})$ of the product system $\mathbb{E}$  as in Theorem \ref{K2}.  Now,  define an isometric covariant representaion $(\sigma, V^{(1)})$ of $E_1$ on $\mathcal{H}$ by 
	$V^{(1)}(\xi)=\Pi^*_{V}M_{\Theta_1}(\xi)\Pi_V,$ then
	it is easy to see that  $(\sigma, V^{(1)},V^{(2)})$
	is  the pure  isometric covariant representaion  of the product system $\mathbb{E}$ on the Hilbert space $\mathcal{H}$ such that $\wV^{(1)} (I_{E_1} \ot \wV^{(2)})=\wt V.$
\end{remark}

\section{Characterization of invariant subspaces for a pure isometric covariant representation of a product system over $\mathbb{N}_0^{2}$}

In this section, we characterize invariant subspaces for a pure isometric covariant representation of product system over $\mathbb{N}_0^{2}$ on a Hilbert space $\mathcal{H}$ (Theorem \ref{DDDD1}), and the BCL-representation is compared with other canonical multi-analytic description of the representations (Theorem \ref{L8}).

Let $(\sigma , V^{(1)},V^{(2)})$ be an isometric covariant representation of $\mathbb{E}$ on $\mathcal{H}.$ Since the   subspaces $\mathcal{W}_i$ and $\mathcal{W}$ are $\sigma$-invariant, let $(\rho_i, S^{\mathcal{W}_i} )$ and $(\rho, S)$ be the induced representaions of $E_i$ and $E$ induced by  the reprsentations $\sigma|_{\mathcal{W}_i}$ and $\sigma|_{\mathcal{W}},$ respectively. Suppose that $(\sigma, V^{(i)})$ is pure, for some $1\leq i \leq 2$ with $i \neq j,$ then define a bi-module map  $V'_j: E_j \rightarrow B(\mathcal{F}({E_i}) \otimes \mathcal{W}_i)$ by $V'_j(\xi_j)=\Pi_{V^{(i)}}V^{(j)}(\xi_j)\Pi^*_{V^{(i)}}, \xi_j \in E_j,$ where $\Pi_{V^{(i)}}$ is the Wold-von Neumann decomposition for $(\sigma, V^{(i)}).$   Then $(\rho_i, V'_j)$ is an isometric covariant representation and by Theorem  \ref{L6}, $V'_j=M_{\Theta_{{V}^{(j)}}}.$ The following theorem analyzes the isometric covariant representation $(\sigma, V^{(1)}, V^{(2)})$ such that  $(\sigma , V^{(i)})$ is pure for some $1 \leq i \leq 2.$

\begin{theorem}\label{L8}
	Suppose that $(\sigma , V^{(1)},V^{(2)})$ is an isometric covariant representation of $\mathbb{E}$ on  $\mathcal{H}$ such that $(\sigma , V^{(i)})$ is pure, for some $1 \leq i \leq 2.$  Then  for $1 \leq j \leq 2$ with $i \neq j,$ the BCL-representation  $(\rho,M_{{\Theta_1}}, M_{{\Theta_2}})$ of  $(\sigma , V^{(1)},V^{(2)})$  is isomorphic (say the unitary isomorphism $\Pi_i$) to  $(\rho_i, S^{\mathcal{W}_i},M_{\Theta_{{V}^{(j)}}})$
	and 
	$$\Pi_{i}\widetilde{S}= \widetilde{S}^{\mathcal{W}_i}(I_{E_i}\ot \widetilde{M}_{\Theta_{{V}^{(j)}}})(\widetilde{t}_{j,i} \ot \Pi_i),$$
	where $\widetilde{t}_{2,1}=I_E$ and $\widetilde{t}_{1,2}=t_{1,2}.$  In particular,  if  for each   $1 \leq i \leq 2,$ $(\sigma , V^{(i)})$ is  pure,  then $(\rho_i, S^{\mathcal{W}_i},M_{\Theta_{{V}^{(j)}}})$ is isomorphic to $(\rho_j, S^{\mathcal{W}_j},M_{\Theta_{{V}^{(i)}}}).$ 
	
\end{theorem}
\begin{proof}
	Fix $1 \leq i \leq 2$ and suppose that $(\sigma , V^{(i)})$ is a pure isometric covariant representation, then $(\sigma , V^{(1)},V^{(2)})$ is also a pure isometric covariant representation. From Theorem \ref{L6}, there exists an isometric covariant representaion $(\rho_i, M_{\Theta_{{V}^{(j)}}})$ of $E_j$ on $\mathcal{F}({E_i}) \otimes \mathcal{W}_i$ such that
	\begin{equation}\label{L7}
		\rho_i(a)\Pi_{V^{(i)}}=\Pi_{V^{(i)}}\sigma(a) \:\:\mbox{and}\:\:\Pi_{V^{(i)}}\wV^{(j)} = \widetilde{M}_{\Theta_{{V}^{(j)}}}(I_{E_j} \ot \Pi_{V^{(i)}}), 
	\end{equation} where 
	${\Theta_{{V}^{(j)}}}(\xi_j)=\sum_{n \in \mathbb{N}_0}\widetilde{S}^{\mathcal{W}_i}_{ n}(I_{E_i^{\otimes n}} \otimes P_{\mathcal{W}_i})\widetilde{V}_{n}^{(i)^*}V^{(j)}(\xi_j), $ for $\xi_j \in E_j$ and  $1 \leq j \leq 2$ with $i \neq j.$

	Let  $(\rho,M_{{\Theta_1}}, M_{{\Theta_2}})$ be the BCL-representation of $(\sigma , V^{(1)},V^{(2)}),$
	then $\Pi_V \wV^{(i)}=\widetilde{M}_{\Theta_i}(I_{E_i} \ot \Pi_V)$ for all $i\in \{1,2\}.$  Define a unitary $\Pi_{i} :  \mathcal{F}({E}) \otimes \mathcal{W} \to \mathcal{F}({E_i}) \otimes \mathcal{W}_i $ by  $\Pi_{i}=\Pi_{V^{(i)}}\Pi_V^* ,$ then 
	$\Pi_i$ satisfies the relations 
	\begin{align*}
		\Pi_{i}\widetilde{M}_{\Theta_i}&=\Pi_{V^{(i)}}\Pi_V^*\widetilde{M}_{\Theta_i}=\Pi_{V^{(i)}}\wV^{(i)}(I_{E_i} \ot \Pi_V^*)=\widetilde{S}^{\mathcal{W}_i}(I_{E_i}\ot  \Pi_{V^{(i)}}\Pi_V^*)=\widetilde{S}^{\mathcal{W}_i}(I_{E_i}\ot  \Pi_{i}) \:\: 
	\end{align*}
	and 	$\Pi_{i}\rho(a)=\Pi_{V^{(i)}}\sigma(a)\Pi_V^*=\rho_i(a)\Pi_{i}.$ Also, using Equation (\ref{L7}), we get 
	\begin{align*}
		\Pi_{i}\widetilde{M}_{\Theta_j}&=\Pi_{V^{(i)}}\Pi_V^*\widetilde{M}_{\Theta_j}=\Pi_{V^{(i)}}\wV^{(j)}(I_{E_j} \ot \Pi_V^*)=\widetilde{M}_{\Theta_{{V}^{(j)}}}(I_{E_j}\ot  \Pi_{V^{(i)}}\Pi_V^*)=\widetilde{M}_{\Theta_{{V}^{(j)}}}(I_{E_j}\ot \Pi_{i}).
	\end{align*}
	This shows that  the BCL-representation  $(\rho,M_{{\Theta_1}}, M_{{\Theta_2}})$  is isomorphic to  $(\rho_i, S^{\mathcal{W}_i},M_{\Theta_{{V}^{(j)}}}), 1 \leq j \leq 2$ with $i \neq j.$
	Since $\Pi_{V^{(i)}}\wV^{(i)}=\widetilde{S}^{\mathcal{W}_i}(I_{E_i}\ot \Pi_{V^{(i)}}),$  the Equation  (\ref{L7}) gives
	\begin{align*}
		\Pi_{V^{(i)}}\wV^{(i)}(I_{E_i} \otimes \wV^{(j)})&=\widetilde{S}^{\mathcal{W}_i}(I_{E_i}\ot \Pi_{V^{(i)}}\wV^{(j)})= \widetilde{S}^{\mathcal{W}_i}(I_{E_i}\ot \widetilde{M}_{\Theta_{{V}^{(j)}}}(I_{E_j} \ot \Pi_{V^{(i)}})).
	\end{align*} Then
	$\Pi_{i}\widetilde{S}=\Pi_{V^{(i)}}\Pi_V^* \widetilde{S}=\Pi_{V^{(i)}} \wV^{(i)}(I_{E_i}\ot \wV^{(j)})(\widetilde{t}_{j,i} \ot  \Pi_V^*)= \widetilde{S}^{\mathcal{W}_i}(I_{E_i}\ot \widetilde{M}_{\Theta_{{V}^{(j)}}})(\widetilde{t}_{j,i} \ot \Pi_i),$ 
	where $\widetilde{t}_{2,1}=I_E$ and $\widetilde{t}_{1,2}=t_{1.2}.$ 
	
	Suppose that for each $1 \leq i \leq 2,$ $(\sigma, V^{(i)})$ is pure.  For $1 \leq j \leq 2$ with $i \neq j,$  define a unitary operator $\Pi_{ij}:  \mathcal{F}({E_i}) \otimes \mathcal{W}_i \rightarrow  \mathcal{F}({E_j}) \otimes \mathcal{W}_j $  by $\Pi_{ij}= \Pi_j \Pi_i^*$. Then $\Pi_{ji}=\Pi^*_{ij}$ and  it is easy to verify that  $\Pi_{ij}\rho_i(a)=\rho_j(a)\Pi_{ij}$ and $ \Pi_{ij}\widetilde{S}^{\mathcal{W}_i}=\widetilde{M}_{{\Theta_{{V}^{(i)}}}}(I_{E_i} \ot \Pi_{ij}).$  Thus  $(\rho_i, S^{\mathcal{W}_i},M_{\Theta_{{V}^{(j)}}})$ is isomorphic to $(\rho_j, S^{\mathcal{W}_j},M_{\Theta_{{V}^{(i)}}}).$ 
\end{proof}

Let $(\sigma , V^{(1)},V^{(2)})$ be a c.b.c. representation of $\mathbb{E}$ on a Hilbert space $\mathcal{H}.$ A non-zero closed subspace $\mathcal{K}$ of $\mathcal{H}$ is called  $(\sigma , V^{(1)},V^{(2)})$-{\it invariant} $(resp. (\sigma , V^{(1)},V^{(2)})$-{\it reducing}) if it is $\sigma$-invariant and (resp. both $\mathcal{K},\mathcal{K}^{\perp}$) is  invariant by each operator  $V^{(i)}(\xi_i), \xi_i \in E_i$ for all $1 \leq i \leq 2.$ We say that $(\sigma , V^{(1)},V^{(2)})$ is  {\it irreducible} if there doesn't exist a non-zero reducing subspace for $(\sigma , V^{(1)},V^{(2)}).$

The following theorem proves a characterization of reducing subspaces for the BCL-repres-\\entation and BCL-triple $(P_{\mathcal{W}_2},U,\mathcal{W})$ (see Theorem \ref{K2} ), which is a generalization of \cite[Lemma 2.1]{DSSS22}.

\begin{theorem}
	Suppose that $(\sigma , V^{(1)},V^{(2)})$ is a pure  isometric covariant representation of $\mathbb{E}$ on $\mathcal{H}$ and $(\rho,M_{{\Theta_1}}, M_{{\Theta_2}})$ is the BCL-representation of $(\sigma , V^{(1)},V^{(2)}).$
	Let $(P_{\mathcal{W}_2},U,\mathcal{W})$ be the BCL-triple and $\mathcal{M} \subseteq \mathcal{F}(E) \otimes \mathcal{W}$ be a closed subspace. Then $\mathcal{M}$ reduces $(\rho, M_{\Theta_1}, M_{\Theta_2})$ if and only if there exists a closed subspace $\mathcal{W'} \subseteq \mathcal{W}$ which is $P_{\mathcal{W}_2}$-invariant  such that $U|_{E_2\ot \mathcal{W}'}:{E_2\ot \mathcal{W}'}\to P_{\mathcal{W}_2}^{\perp} \mathcal{W}'\oplus E_2 \ot E_1 \ot P_{\mathcal{W}_2}\mathcal{W}'$ is a unitary isomorphism and  $$\mathcal{M} = \mathcal{F}(E) \otimes \mathcal{W}'.$$
\end{theorem}
\begin{proof}
	Suppose that $\mathcal{M}$ reduces $(\rho, M_{\Theta_1}, M_{\Theta_2})$, then $\mathcal{M}$ reduces $(\rho,{S})$ (as $\widetilde{S}=	\widetilde{M}_{{\Theta}_1}(I_{E_1}\ot \widetilde{M}_{{\Theta}_2})$), and thus, there exists a closed subspace $\mathcal{W'} \subseteq \mathcal{W}$ such that $\mathcal{M} = \mathcal{F}(E) \otimes \mathcal{W}',$ where  $\mathcal{W}'$ is the generating wandering subspace for $(\rho,{S})|_{\mathcal{M}}$. Define a closed subspace  $\mathcal{K}$  of $\mathcal{H}$ by  $$\mathcal{K}=\bigoplus_{n\geq 0} \wV_n(E^{\ot n}\ot \mathcal{W}'),$$ then $\mathcal{W}'=\mathcal{K}\ominus \widetilde{V}(E\ot \mathcal{K}).$ Since $(\sigma , V^{(1)},V^{(2)})$ is isomorphic to $(\rho, M_{\Theta_1}, M_{\Theta_2}),$  $\mathcal{K}$ is reducing for $(\sigma , V^{(1)},V^{(2)}).$ It follows that $\mathcal{K}$ reduces $(\sigma,V),$ and $\mathcal{W}'$ is the generating wandering subspace for $(\sigma,V)|_{\mathcal{K}}.$ From Equation (\ref{WWW}),   $\mathcal{W}'=\mathcal{W}'_{1} \oplus \widetilde{V}^{(1)}(E_1 \otimes \mathcal{W}'_{2})=\widetilde{V}^{(2)}(E_2 \otimes \mathcal{W}'_{1}) \oplus \mathcal{W}'_{2}.$ Thus $U|_{E_2\ot \mathcal{W}'}:{E_2\ot \mathcal{W}'}\to \wV^{(2)}(E_2 \ot \mathcal{W}'_1) \oplus E_2 \ot E_1 \ot \mathcal{W}'_2$ is a unitary isomorphism and $P_{\mathcal{W}_2}\mathcal{W}'=\mathcal{W}'_2\subseteq \mathcal{W}'.$ Therefore $\mathcal{W}'$ is an invariant subspace for $P_{\mathcal{W}_2}.$ 
	
	Conversely, suppose $\mathcal{W}'$ is a closed subspace of $\mathcal{W}$ such that $\mathcal{M} = \mathcal{F}(E) \otimes \mathcal{W}',$ $\mathcal{W}'$ is invariant for $P_{\mathcal{W}_2}$ and $U|_{E_2\ot \mathcal{W}'}:{E_2\ot \mathcal{W}'}\to P_{\mathcal{W}_2}^{\perp} \mathcal{W}'\oplus E_2 \ot E_1 \ot P_{\mathcal{W}_2}\mathcal{W}'$ is a unitary isomorphism. Then  $\mathcal{M}$ reduces $(\rho,S).$ From Theorem \ref{K2}, we get 
	\begin{equation}\label{AAA1}
		\widetilde{\Theta}_2(E_2 \ot \mathcal{W}')\subseteq \mathcal{F}(E) \otimes \mathcal{W}',
	\end{equation}
	\begin{equation}\label{AAA12}
		\widetilde{\Theta}_2^*( P_{\mathcal{W}_2}^{\perp} \mathcal{W}'\oplus E \ot P_{\mathcal{W}_2}\mathcal{W}')\subseteq E_2\ot \mathcal{W}' \quad and \quad \widetilde{\Theta}_2^*( P_{\mathcal{W}_2} \mathcal{W}')=0.
	\end{equation}
	Therefore by Equation (\ref{I_1}), $\mathcal{M}$ reduces $(\rho,M_{{\Theta_2}}).$ 
	Since $\widetilde{M}_{{\Theta}_1}(I_{E_1}\ot \widetilde{M}_{{\Theta}_2})=\widetilde{S}$ and the range of $\widetilde{M}_{{\Theta}_2}$ equals $P_{\mathcal{W}_2}^{\perp}\mathcal{W}\oplus E\ot \mathcal{F}(E) \otimes \mathcal{W},$ we have $\widetilde{M}_{\Theta_1}(E_1\ot (P_{\mathcal{W}_2}^{\perp}\mathcal{W}'\oplus E\ot \mathcal{F}(E) \otimes \mathcal{W}'))\subseteq E\ot \mathcal{F}(E) \otimes \mathcal{W}'$ and $\widetilde{M}_{\Theta_1}^*(E\ot \mathcal{F}(E) \otimes \mathcal{W}')\subseteq E_1\ot \mathcal{F}(E) \otimes \mathcal{W}'.$ We want to prove that $\widetilde{M}_{\Theta_1}(E_1\ot P_{\mathcal{W}_2}\mathcal{W}')\subseteq \mathcal{M}.$ By the definition of $\widetilde{\Theta}_1$ we get \begin{align*}
		\widetilde{M}_{\Theta_1}(E_1\ot P_{\mathcal{W}_2}\mathcal{W}')=\widetilde{\Theta}_1(E_1\ot P_{\mathcal{W}_2}\mathcal{W}')=\wV^{(1)}(E_1\ot P_{\mathcal{W}_2}\mathcal{W}')\subseteq \mathcal{W}'.
	\end{align*} Since $\mathcal{W}'=(\mathcal{W}_{1}\cap \mathcal{W}' ) \oplus (\widetilde{V}^{(1)}(E_1 \otimes \mathcal{W}_{2}) \cap \mathcal{W}'),$ we have \begin{align*}
		\widetilde{M}_{\Theta_1}^*(\mathcal{W}')=\widetilde{M}_{\Theta_1}^*(\widetilde{V}^{(1)}(E_1 \otimes \mathcal{W}_{2}) \cap \mathcal{W}')\subseteq E_1\ot \mathcal{W}'.
	\end{align*} Then $\mathcal{M}$ reduces $(\rho, M_{\Theta_1}),$ and hence $\mathcal{M}$ reduces $(\rho, M_{\Theta_1}, M_{\Theta_2}).$
\end{proof}

\begin{remark}
	Let $(P_{\mathcal{W}_2},U,\mathcal{W})$ be the BCL-triple, and $\mathcal{M} \subseteq \mathcal{F}(E) \otimes \mathcal{W}$ be a closed subspace. Then   $(\rho, M_{\Theta_1}, M_{\Theta_2})$ is irreducible if and only if there doesn't exist a non-zero closed subspace $\mathcal{W'} \subseteq \mathcal{W}$ which is $P_{\mathcal{W}_2}$-invariant such that $U|_{E_2\ot \mathcal{W}'}:{E_2\ot \mathcal{W}'}\to P_{\mathcal{W}_2}^{\perp} \mathcal{W}'\oplus E_2 \ot E_1 \ot P_{\mathcal{W}_2}\mathcal{W}'$ is a unitary isomorphism and  $\mathcal{M} = \mathcal{F}(E) \otimes \mathcal{W}'.$
\end{remark}

\begin{definition}
	Let $(\sigma , V^{(1)},V^{(2)})$ and  $(\sigma' , V^{(1)'},V^{(2)'})$ be c.b.c. representations of $\mathbb{E}$ on $\mathcal{H}$ and $\mathcal{H}',$ respectively.  A bounded operator $B: \mathcal{H} \rightarrow \mathcal{H}'$ is said to be {\rm multi-analytic} for  the covariant representations $(\sigma , V^{(1)},V^{(2)})$ and  $(\sigma' , V^{(1)'},V^{(2)'})$  if $B$ intertwine  these two representations, that is, $$B\sigma(a)=\sigma'(a)B \quad \mbox{ and}  \quad BV^{(i)}(\xi_i)=V^{(i)'}(\xi_i)B, \quad a \in \mathcal{A}, \xi_i \in E_i , 1 \leq i \leq 2.$$  
\end{definition}

Suppose that $(\sigma , V^{(1)},V^{(2)})$ and  $(\sigma' , V^{(1)'},V^{(2)'})$ are the  isometric covariant representations of $\mathbb{E}$ on $\mathcal{H}$ and $\mathcal{H}',$ respectively, such that for some $1 \leq i \leq 2,$ $(\sigma, V^{(i)})$ and $(\sigma', V^{(i)'})$ are pure. Then by Theorem \ref{L4},

$$\mathcal{H}=\bigoplus_{{n} \in \mathbb{N}_0}\wt{V}_n^{(i)}(E_i^{\ot n} \otimes \mathcal{W}_i ) \:\: \mbox{and} \:\: \mathcal{H}'=\bigoplus_{{n} \in \mathbb{N}_0}\wt{V}_n^{(i)'}(E_i^{\ot n} \otimes \mathcal{W}_i'),$$ where $\mathcal{W}_i$ and $\mathcal{W}'_i$ are the generating wandering subspaces for $(\sigma , V^{(i)})$ and  $(\sigma' , V^{(i)'}),$ respectively. Let  $B: \mathcal{H} \to \mathcal{H}'$ be a multi-analytic operator, then $B$ is uniquely determined by  the operator $\Psi: \mathcal{W}_i \to \mathcal{H}'$ satisfying $\Psi \sigma(a)h=\sigma'(a)\Psi(h), h \in \mathcal{W}_i$ where $ \Psi= B|_{\mathcal{W}_i}.$ Indeed,  for $h \in \mathcal{W}_i, \eta_{{n}} \in E_i^{\ot n}$ we have $BV^{(i)}_{{n}}(\eta_{{n}})h=V_{n}^{(i)'}(\eta_{{n}})\Psi h$ and $\mathcal{H}=\bigoplus_{{n} \in \mathbb{N}_0}\wt{V}^{(i)}_{n}(E_i^{\ot n} \otimes \mathcal{W}_i ).$  Suppose that $\Psi: \mathcal{W}_i \to \mathcal{H}'\left( = \bigoplus_{{n} \in \mathbb{N}_0}\wt{V}_n^{(i)'}(E_i^{\ot n} \otimes \mathcal{W}_i') \right)$ is an operator which satisfies $\sigma'(a)\Psi h=\Psi \sigma(a)h, h\in \mathcal{W}_i.$  Define an operator $M_{\Psi}: \mathcal{H} \to \mathcal{H}'$ by  $$M_{\Psi}V_n^{(i)}(\eta_{{n}})h=V_{{n}}'(\eta_{{n}})\Psi h=V_n^{(i)'}(\eta_{{n}})M_{\Psi} h \:\:\: \eta_{{n}} \in E_i^{\ot n}, h \in \mathcal{W}_i,$$ 
where $\mathcal{H}=\bigoplus_{{n} \in \mathbb{N}_0}\wt{V}^{(i)}_{n}(E_i^{\ot n} \otimes \mathcal{W}_i ).$ Then $M_{\Psi}$ is  multi-analytic   for $(\sigma, V^{(i)})$ and  $(\sigma', V^{(i)'}),$ but it is not multi-analytic for  $(\sigma, V^{(j)})$ and  $(\sigma', V^{(j)'})$ for $1 \leq j \leq 2$ with $i \neq j.$ It happens if $$\mathcal{H}=\bigoplus_{\mathbf{n} \in \mathbb{N}^2_0}\wt{V}_{\mathbf{n}}(\mathbb{E}(\mathbf{n}) \otimes \mathcal{W}) \:\: \mbox{and} \:\: \mathcal{H}'=\bigoplus_{\mathbf{n} \in \mathbb{N}^2_0}\wt{V}'_{\mathbf{n}}(\mathbb{E}(\mathbf{n}) \otimes \mathcal{W}'),$$   $(\sigma , V^{(1)},V^{(2)})$ and  $(\sigma' , V^{(1)'},V^{(2)'})$ are doubly commuting isometries such that  $(\sigma , V^{(i)})$ and  $(\sigma' , V^{(i)'})$ are pure, for all $1 \leq i \leq 2$  (see \cite[Section 4]{TV21}). This shows that $M_{\Psi}$ is multi-analytic for $(\sigma , V^{(1)},V^{(2)})$ and  $(\sigma' , V^{(1)'},V^{(2)'}).$ 

An operator $\Psi: \mathcal{W}_i \to \mathcal{H}'$ such that $\sigma'(a)\Psi h=\Psi \sigma(a)h, h\in \mathcal{W}_i,$ 
is {\it inner} if $M_{\Psi}$ is an isometry. Note that $\Psi$ is inner if and only if $\Psi$ is an isometry and $\Psi(\mathcal{W}_i)$ is a wandering subspace for $(\sigma', V^{(i)'}).$

The following theorem is a characterization of the invariant subspaces for pure isometric covariant representation of a product system over $\mathbb{N}_0^{2}$
which is an analogue of \cite[Theorem 3.2]{MMSS19} and \cite[Theorem 4.4]{TV21}.
\begin{theorem}\label{DDDD1}
	Let $(\sigma , V^{(1)},V^{(2)})$ be an isometric covariant representation of $\mathbb{E}$ on $\mathcal{H}$ such that  $(\sigma,V^{(i)})$ is pure for some $1\le i \le 2$. Let $\mathcal{M}\subseteq\mathcal{F}({E}_i)\ot \mathcal{W}_i$ be a closed subspace. Then $\mathcal{M}$ is invariant for $(\rho_i,S^{\mathcal{W}_i}, {M}_{\Theta_{{V}^{(j)}}})$ if and only if there exist a Hilbert space $\mathcal{K}$, an isometric covariant represenation $(\varrho_i , V^{(i)'},V^{(j)'})$, $1\le j \le 2$ with $i\neq j$, of $\mathbb{E}$ on $\mathcal{K}$ and an inner operator $\Psi:\mathcal{W}'_i\to \mathcal{W}_i$  for $(\varrho_i , V^{(i)'})$ and  $(\rho_i,S^{\mathcal{W}_i})$ such that $(\varrho_i , V^{(i)'})$ is pure, $M_{\Psi}\widetilde{V}^{(j)'}= \widetilde{M}_{\Theta_{{V}^{(j)}}}(I_{E_j} \ot M_{\Psi})$ and  $$\mathcal{M}=M_{\Psi}\mathcal{K},$$ where $\mathcal{W}_i$ and $\mathcal{W}'_i$  are the generating wandering subspaces for $(\sigma, V^{(i)})$ and $(\rho_i,S^{\mathcal{W}_i})|_{\mathcal{M}},$ respectively. In particular, $M_{\Psi}$ is multi-analytic for  $(\varrho_i , V^{(i)'},V^{(j)'})$ and  $(\rho_i,S^{\mathcal{W}_i}, {M}_{\Theta_{{V}^{(j)}}}).$
\end{theorem}
\begin{proof}
	Fix $1\le i \le 2$ and let $1\le j \le 2$ with $i\neq j$. Let $\mathcal{M}\subseteq\mathcal{F}({E}_i)\ot \mathcal{W}_i$ be a closed invariant subspace for $(\rho_i,S^{\mathcal{W}_i}, {M}_{\Theta_{{V}^{(j)}}})$ of ${\mathbb{E}}$ on $\mathcal{F}({E}_i)\ot \mathcal{W}_i.$ Define an isometric covariant representation $(\pi_i, T^{(i)}, T^{(j)}) = (\rho_i,S^{\mathcal{W}_i}, {M}_{\Theta_{{V}^{(j)}}})|_{\mathcal{M}}$ of $\mathbb{E}$ on a Hilbert space $\mathcal{M}$. Clearly $(\pi_i, T^{(i)}, T^{(j)})$ is  pure, since $(\pi_i, T^{(i)})$ is pure. Let $\Pi_{T^{(i)}}:\mathcal{M}\to \mathcal{F}(E_i)\ot \mathcal{W}_i'$ be the Wold-von Neumann decomposition of $(\pi_i,T^{(i)})$ on $\mathcal{M},$ where $\mathcal{W}_i'=\mathcal{M}\ominus \wT^{(i)}(E_i\ot \mathcal{M})$ is a generating wandering
	subspace for $(\pi_i,T^{(i)})$. That is, $\Pi_{T^{(i)}}$ is unitary and satisfies 
	\begin{equation}
		\Pi_{T^{(i)}}\pi_i(a)=\rho'_i(a)\Pi_{T^{(i)}} \quad and \quad	\Pi_{T^{(i)}}\widetilde{T}^{(i)}=\widetilde{S}^{\mathcal{W}_i'}(I_{E_i }\ot \Pi_{T^{(i)}}), \quad a\in \mathcal{A},
	\end{equation} where $(\rho'_i,S^{\mathcal{W}_i'})$ is the induced representation induced by $\pi_i|_{\mathcal{W}'_i}.$ Then by Theorem \ref{L6},
	there exists an isometric bi-module map $\Theta_{T^{(j)}}:E_{j}\to B(\mathcal{W}_i',\mathcal{F}({E}_i)\ot \mathcal{W}_i')$ such that
	\begin{equation}
		\Pi_{T^{(i)}}\widetilde{T}^{(j)}=\widetilde{M}_{\Theta_{T^{(j)}}}(I_{E_j}\ot \Pi_{T^{(i)}}),
	\end{equation}
	where $\Theta_{T^{(j)}}(\xi_j)=\sum_{n \in \mathbb{N}_0}\widetilde{S}_{n}^{\mathcal{W}'_i}(I_{E_i^{\otimes n}} \otimes P_{\mathcal{W}_i'})\widetilde{T}_{n}^{(i)^*}T^{(j)}(\xi_j)|_{\mathcal{W}_i'}, ~\xi_j \in E_{j}.$ Therefore $(\rho'_i,S^{\mathcal{W}_i'},{M}_{\Theta_{T^{(j)}}})$ is an isometric covariant representation of $\mathbb{E}$ on $\mathcal{F}({E}_i)\ot \mathcal{W}_i'.$
	
	Let $i_{\mathcal{M}}$ denotes the inclusion map from ${\mathcal{M}}$ to $\mathcal{F}({E}_i)\ot \mathcal{W}_i.$ Define an isometry $\Pi_{\mathcal{M}}: \mathcal{F}(E_i)\ot \mathcal{W}_i'\to \mathcal{F}({E}_i)\ot \mathcal{W}_i$ by $\Pi_{\mathcal{M}}=i_{\mathcal{M}}\Pi_{T^{(i)}}^*,$ then $\Pi_{\mathcal{M}}\Pi_{\mathcal{M}}^*=i_{\mathcal{M}}i_{\mathcal{M}}^*$ and hence the range of $\Pi_{\mathcal{M}}$ equals ${\mathcal{M}}.$ Note that $i_{\mathcal{M}}\pi_i(a)=\rho_i(a)i_{\mathcal{M}},$ $i_{\mathcal{M}}\wT^{(i)}=\widetilde{S}^{\mathcal{W}_i}(I_{E_i}\ot i_{\mathcal{M}})$ and $i_{\mathcal{M}}\wT^{(j)}=\widetilde{M}_{\Theta_{{V}^{(j)}}}(I_{E_{j}}\ot i_{\mathcal{M}}),$ we have $$\Pi_{\mathcal{M}}\rho'_i(a)=\rho_i(a)\Pi_{\mathcal{M}},\quad
	\Pi_{\mathcal{M}}\widetilde{S}^{\mathcal{W}_i'}=i_{\mathcal{M}}\wT^{(i)}(I_{E_i}\ot \Pi_{T^{(i)}}^*)=\widetilde{S}^{\mathcal{W}_i}(I_{E_i}\ot \Pi_{\mathcal{M}})$$ and \begin{equation*}
		\Pi_{\mathcal{M}}\widetilde{M}_{\Theta_{T^{(j)}}}=i_{\mathcal{M}}\wT^{(j)}(I_{E_{j}}\ot \Pi_{T^{(i)}}^*)=\widetilde{M}_{\Theta_{{V}^{(j)}}}(I_{E_{j}}\ot \Pi_{\mathcal{M}}).
	\end{equation*} Then $\Pi_{\mathcal{M}}$ is a multi-analytic operator from $(\rho'_i,S^{\mathcal{W}_i'}, {M}_{\Theta_{{T}^{(j)}}})$ to $(\rho_i,S^{\mathcal{W}_i}, {M}_{\Theta_{{V}^{(j)}}}),$ and hence there exists an inner operator $\Psi:\mathcal{W}_i'\to \mathcal{W}_i$ such that $\Pi_{\mathcal{M}}=M_{\Psi}$. Therefore $$\mathcal{M}=\Pi_{\mathcal{M}}(\mathcal{F}(E_i)\ot \mathcal{W}_i')=M_{\Psi}(\mathcal{F}(E_i)\ot \mathcal{W}_i').$$
	The converse part is obvious.
\end{proof}

\section{Connection between the Defect Operator and Fringe Operators for an isometric covariant representation of a product system over $\mathbb{N}_0^{2}$}
In this section, we will introduce the notion of \textit{Fringe operators} and \textit{joint defect operator}, and discuss the relationship between them. Define the \textit{joint defect operator} $C(\sigma,V^{(1)}, V^{(2)})$ (or, simply $C$) of the isometric covariant representation $(\sigma , V^{(1)},V^{(2)})$ of $\mathbb{E}$ on $\mathcal{H}$ by
$$C(\sigma,V^{(1)}, V^{(2)})=I-\wV^{(1)}\wV^{(1)^*}-\wV^{(2)}\wV^{(2)^*}+\wV\wV^*,$$ where $\wV=\widetilde{V}^{(1)}(I_{E_1} \otimes \widetilde{V}^{(2)} ).$ From Equation $(\ref{WWW}),$ we get \begin{align}\label{L2}
	P_{\mathcal{W}} = P_{\mathcal{W}_1} \oplus P_{\wV^{(1)} (E_1 \ot \mathcal{W}_2)} = P_{\wV^{(2)} (E_2 \ot \mathcal{W}_1)} \oplus
	P_{\mathcal{W}_2}.
\end{align}
This implies that \begin{align}\label{L3}
	C&=P_{\mathcal{W}_1} + P_{\mathcal{W}_2} - P_{\mathcal{W}}=P_{\mathcal{W}_1} - P_{\wV^{(2)} (E_2 \ot \mathcal{W}_1)} = P_{\mathcal{W}_2}-P_{\wV^{(1)} (E_1 \ot \mathcal{W}_2)}.
\end{align}

\begin{lemma}\label{L5}
	Let $(\sigma , V^{(1)},V^{(2)})$ be  an isometric covariant representation of $\mathbb{E}$ on $\mathcal{H}$.  If $\mathcal{H}=\mathcal{H}_s(V) \oplus \mathcal{H}_u(V)$ is the Wold-decomposition of $(\sigma, V),$ then the subspaces $\mathcal{H}_s(V)$ and $\mathcal{H}_u(V)$ are $(\sigma , V^{(i)})$-reducing, $$\mathcal{H}_s(V^{(i)}) \subseteq \mathcal{H}_s(V) \quad and \quad \mathcal{H}_u(V) \subseteq \mathcal{H}_u(V^{(i)}), \quad  1 \leq i \leq 2.$$ 
\end{lemma}
\begin{proof}
	Observe that  $\mathcal{H}_u(V) \subseteq \mathcal{H}_u(V^{(i)})$ and thus $\mathcal{H}_s(V^{(i)}) \subseteq \mathcal{H}_s(V).$  Therefore, it is enough to show that $\mathcal{H}_s(V)$  reduces $(\sigma , V^{(i)}), 1\leq i \leq 2.$  From Equation $(\ref{WWW})$ and for $1\leq j \leq2 $ with $i \neq j,$ we get
	\begin{align}\label{WW} 
		\wV^{(i)}(E_i \ot \mathcal{W}) \subseteq  \wV^{(i)}(E_i \ot \mathcal{W}_j)  \oplus \wV({E} \ot \mathcal{W}_i) \subseteq   \mathcal{W} \oplus \wV({E} \ot \mathcal{W}).
	\end{align}
	For each $\xi_i \in E_i$ and $\eta\in E\ot \mathcal{W},$ we have
	\begin{align*}
		V^{(i)}(\xi_i) \wV(\eta)&= \wV^{(i)}(I_{E_i}\ot \wV)(\xi_i\ot \eta )=\wV(I_{E} \ot \wV^{(i)})(t_{i,{\bf{1}}}\ot I_{\mathcal{W}})(\xi_i\ot \eta )\\& \in \wV({E} \ot \wV^{(i)}(E_i \ot \mathcal{W})) \subseteq \wV(E \ot \mathcal{W}) \oplus \wV_2 (E^{\ot 2} \ot \mathcal{W})\subseteq \mathcal{H}_s(V).
	\end{align*}
	Continuing in this way, we get 
	$V^{(i)}(\xi_i) \wV_n(E^{\ot n}\ot \mathcal{W}) \subseteq \wV_n(E^{\ot n}\ot \mathcal{W})  \oplus \wV_{n+1}(E^{\ot n+1}\ot \mathcal{W}) \subseteq  \mathcal{H}_s(V),$ for $n\ge 0.$  Since  $\mathcal{H}_s(V)$ is $\sigma$-invariant,  the subspace  $\mathcal{H}_s(V)$ is $(\sigma,V^{(i)})$-invariant.  Observe  that  $\wV^{(i)^*}\mathcal{W}=E_i \ot \mathcal{W}_j \subseteq E_i \ot \mathcal{W}$ and  by using Equation (\ref{WW}),  we obtain
	\begin{align*}
		\wV^{(i)^*}\wV(E\ot \mathcal{W})&=(I_{E_i}\ot \wV^{(j)})(t_{j,i}\ot I_{\mathcal{W}})(E\ot \mathcal{W})\subseteq E_i \ot \wV^{(j)}(E_j \ot \mathcal{W})\\& \subseteq E_i \ot (\mathcal{W} \oplus \wV(E \ot \mathcal{W})) \subseteq E_i \ot \mathcal{H}_s(V).
	\end{align*} Similarly,  we can easily prove that
	$\wV^{(i)^*}\wV_n(E^{\ot n}\ot \mathcal{W})\subseteq E_i \ot (\wV_{n-1}(E^{\ot n-1}\ot \mathcal{W}) \oplus \wV_n(E^{\ot n}\ot \mathcal{W})) \subseteq E_i \ot \mathcal{H}_s(V)$ for each $n\geq 1.$  This shows that  $\mathcal{H}_s(V)$  reduces $(\sigma , V^{(i)})$ and hence $\mathcal{H}_u(V)$ is also  $(\sigma , V^{(i)})$-reduces   for  $1 \leq i \leq 2.$ 
\end{proof}

\begin{definition}  
	A c.b.c. representation  $(\sigma, V^{(1)}, V^{(2)})$  of  $\mathbb{E}$ on $\mathcal H$  is said to be {\rm doubly
		commuting} (cf. \cite{S08}) if 
	\begin{equation}
		\wV^{(1)^*} \wV^{(2)} =
		(I_{E_1} \ot \wV^{(2)})  (t_{2,1} \ot I_{\mathcal H})  (I_{E_2} \ot \wV^{(1)^*}).
	\end{equation}
\end{definition}
Suppose that $(\sigma, V^{(1)}, V^{(2)})$ is an isometric covariant representation of  $\mathbb{E}$  on $\mathcal H.$  Then by  Equation (\ref{L3}),  $\mathcal{W}_1$ is an invariant subspace for $(\sigma,V^{(2)})$ if and only if $\mathcal{W}_2$ is an invariant subspace for $(\sigma,V^{(1)}).$
The following lemma shows that the relation between doubly commuting isometric covariant representation and the wandering subspace $\mathcal{W}_i$ for $(\sigma, V^{(i)}).$  We denote by $R(\wV)$ and $N(\wV)$, the range of $\wV$ and the kernel of $\wV,$ respectively. 
\begin{lemma}\label{T1}
	Suppose that $(\sigma , V^{(1)},V^{(2)})$ is an isometric covariant representation of $\mathbb{E}$ on $\mathcal{H}$. The following conditions are equivalent:
	\begin{enumerate}
		\item $(\sigma , V^{(1)},V^{(2)})$ is doubly commuting;
		\item   for $i \neq j, \mathcal{W}_i$ is an invariant subspace for $(\sigma,V^{(j)}).$
	\end{enumerate} 
\end{lemma}
\begin{proof}
	Fix $1 \leq i,j \leq 2$ with $i \neq j.$ 
	
	$(1) \Rightarrow (2)$ Let $\eta_j \in E_j \ot \mathcal{W}_i$ and by doubly commutivity of $(\sigma , V^{(1)},V^{(2)})$ we have 
	\begin{align*}
		\wV^{(i)^*} \wV^{(j)}(\eta_j)&=(I_{E_i} \ot \wV^{(j)})  (t_{j,i} \ot I_{\mathcal H})  (I_{E_j} \ot \wV^{(i)^*})( {\eta_j})=0.
	\end{align*} Therefore ${\wV^{(j)} (E_j \ot \mathcal{W}_i)} \subseteq \mathcal{W}_i \:(= N(\wV^{(i)^*})),$ that is,  $ \mathcal{W}_i$ is an invariant subspace for $(\sigma,V^{(j)}).$ 
	
	$(2) \Rightarrow (1)$ From Theorem \ref{L4}, let $\mathcal{H}=\mathcal{H}_s(V) \oplus \mathcal{H}_u(V)$ be the Wold-decomposition of $(\sigma, V),$ then by Lemma \ref{L5} the subspaces $\mathcal{H}_s(V)$ and $\mathcal{H}_u(V)$ are $(\sigma , V^{(1)},V^{(2)})$-reducing and $\mathcal{H}_u(V) \subseteq \mathcal{H}_u(V^{(i)}), 1 \leq i \leq 2.$ This shows that $(\sigma , V^{(1)},V^{(2)})|_{\mathcal{H}_u(V)}$ is an isometric as well as fully co-isometric covariant representation of $\mathbb{E}$ on $\mathcal{H}_u(V).$  Since $\wV^{(1)^*}\wV^{(1)}=I$ on ${E_1\ot \mathcal{H}_u(V)}$ and $\wV^{(2)}(E_2\ot \mathcal{H}_u(V))\subseteq \mathcal{H}_u(V),$ we get $$\wV^{(1)^*}\wV^{(1)}(I_{E_1}\ot \wV^{(2)})=(I_{E_1}\ot \wV^{(2)})\quad on \quad {E_1\ot E_2\ot \mathcal{H}_u(V)}.$$ Thus $\wV^{(1)^*}\wV^{(2)}(I_{E_2}\ot \wV^{(1)})(t_{1,2}\ot I_{\mathcal{H}_u(V)})=(I_{E_1}\ot \wV^{(2)})$ on ${E_1\ot E_2\ot \mathcal{H}_u(V)}.$ Since $(\sigma ,V^{(1)})|_{\mathcal{H}_u(V)}$ is fully co-isometric, we obtain $$\wV^{(1)^*}\wV^{(2)}=(I_{E_1}\ot \wV^{(2)})(t_{2,1}\ot I_{\mathcal{H}_u(V)})(I_{E_2}\ot \wV^{(1)^*}) \quad on \quad {E_1\ot E_2\ot \mathcal{H}_u(V)}.$$ Therefore $(\sigma , V^{(1)},V^{(2)})|_{\mathcal{H}_u(V)}$  is  doubly commuting. Now we need to  show that   $(\sigma , V^{(1)},V^{(2)})|_{\mathcal{H}_s(V)}$ is doubly commuting. 
	For $n \in \mathbb{N},$  consider the equation 
	\begin{align*}
		&(\wV^{(i)^*} \wV^{(j)}-(I_{E_i} \ot \wV^{(j)})  (t_{j,i} \ot I_{\mathcal H})  (I_{E_j} \ot \wV^{(i)^*}))(I_{E_j} \ot \wV_n)\\&= (\wV^{(i)^*} \wV^{(j)}(I_{E_j}\ot \wV)-(I_{E_i} \ot \wV^{(j)})  (t_{j,i} \ot I_{\mathcal H})  (I_{E_j} \ot \wV^{(i)^*})(I_{E_j}\ot \wV)) (I_{E_j \ot E} \ot  \wV_{n-1})\\& =((I_{E_i}\ot \wV^{(j)})(t_{j,i} \ot \wV^{(j)})-(I_{E_i} \ot \wV^{(j)})  (t_{j,i} \ot\wV^{(j)}))(I_{E_j \ot E} \ot  \wV_{n-1})=0.
	\end{align*} 
	Then from Equation $(\ref{WWW})$ and by hypothesis $(2)$,  it is easy to verify that $\wV^{(i)^*} \wV^{(j)} =
	(I_{E_i} \ot \wV^{(j)})  (t_{j,i} \ot I_{\mathcal H})  (I_{E_j} \ot \wV^{(i)^*})$ on $E_j \ot \mathcal{W}.$  
	Hence,  it follows from the  above observations and by the definition of $\mathcal{H}_s(V),$ the representation $(\sigma , V^{(1)},V^{(2)})|_{\mathcal{H}_s(V)}$  is  doubly commuting.
\end{proof}

Suppose that $(\sigma , V^{(1)},V^{(2)})$ is an isometric covariant representation of $\mathbb{E}$ on $\mathcal{H}.$ Define the \textit{Fringe operators}  $\widetilde{F}^{(1)}: E_1 \ot \mathcal{W}_2 \to \mathcal{W}_2$ and $\widetilde{F}^{(2)}: E_2 \ot \mathcal{W}_1 \to \mathcal{W}_1$ by   
\begin{align}\label{H4}
	\widetilde{F}^{(1)}=P_{\mathcal{W}_2} \wV^{(1)}|_{E_1 \ot \mathcal{W}_2}\quad and \quad \widetilde{F}^{(2)}=P_{\mathcal{W}_1} \wV^{(2)}|_{E_2 \ot \mathcal{W}_1}.
\end{align} 
Since $\mathcal{W}_i$ is $\sigma$-invariant, clearly  $(\sigma_i,F^{(i)})$ is a covariant representation of $E_i$ on $\mathcal{W}_j, $   where $\sigma_i=\sigma|_{\mathcal{W}_j},1\leq j \leq 2$ with $i \neq j.$
Then 
\begin{align*}	I_{E_i \ot \mathcal{W}_j} - \widetilde{F}^{(i)^*}\widetilde{F}^{(i)}& = I_{E_i \ot \mathcal{W}_j} - P_{E_i \ot \mathcal{W}_j} \wV^{(i)^*} P_{\mathcal{W}_j}\wV^{(i)}|_{E_i \ot \mathcal{W}_j},
	\\&= P_{E_i \ot \mathcal{W}_j} \wV^{(i)^*} P_{\wV^{(j)}({E_j \ot \mathcal{W}_i}) }\wV^{(i)}|_{E_i \ot \mathcal{W}_j},
\end{align*} 
where the last Equality follows from  Equation  (\ref{L2}).  Therefore
\begin{align}\label{T2}
	\widetilde{F}^{(i)^*}\widetilde{F}^{(i)} =I_{E_i \ot \mathcal{W}_j} \Leftrightarrow      P_{\wV^{(j)}({E_j \ot \mathcal{W}_i}) }\wV^{(i)}|_{E_i \ot \mathcal{W}_j}=0 \Leftrightarrow   {\wV^{(i)} (E_i \ot \mathcal{W}_j)} \subseteq   \mathcal{W}_j ,\end{align}for $1 \leq i,j \leq 2$ with $i \neq j.$ That is,  $(\sigma_i, F^{(i)})$ is an isometric covariant representation of $E_i$ on $\mathcal{W}_j$ if and only if $\mathcal{W}_j$ is an invariant subspace  for $(\sigma_i, V^{(i)}).$
\begin{notation}
	$[\wV^{(i)^*},\wV^{(j)}]=\wV^{(i)^*}\wV^{(j)}-(I_{E_i}\ot \wV^{(j)})(t_{j,i}\ot I_{\mathcal{H}})(I_{E_j}\ot \wV^{(i)^*})$ for all $1 \leq  i,j \leq 2.$
\end{notation}
Now we discuss the connection between the Fringe operators and $[\wV^{(i)^*},\wV^{(j)}],P_{\mathcal{W}_i},P_{\mathcal{W}_j}$ for $1 \leq  i,j \leq 2.$ Note that
\begin{align}\label{K7}
	\nonumber	\widetilde{F}^{(i)^*}\widetilde{F}^{(i)}&=(I_{E_i}\ot P_{\mathcal{W}_j})\wV^{(i)^*}P_{\mathcal{W}_j}\wV^{(i)}|_{E_i \ot \mathcal{W}_j}=(I_{E_i}\ot P_{\mathcal{W}_j})(I_{E_i\ot \mathcal{W}_j}-\wV^{(i)^*}\wV^{(j)}\wV^{(j)^*}\wV^{(i)})|_{E_i \ot \mathcal{W}_j}\\&
	\nonumber = I_{E_i \ot \mathcal{W}_j}-(I_{E_i}\ot P_{\mathcal{W}_j})\wV^{(i)^*}\wV^{(j)}([\wV^{(j)^*},\wV^{(i)}] +(I_{E_j}\ot\wV^{(i)})(t_{i,j}\ot I_{\mathcal{H}})(I_{E_i}\ot\wV^{(j)^*}))\\&
	=I_{E_i \ot \mathcal{W}_j}-(I_{E_i}\ot P_{\mathcal{W}_j})\wV^{(i)^*}\wV^{(j)}[\wV^{(j)^*},\wV^{(i)}]
	=I_{E_i \ot\mathcal{W}_j}-[\wV^{(i)^*},\wV^{(j)}][\wV^{(j)^*},\wV^{(i)}]|_{E_i \ot \mathcal{W}_j}.
\end{align}
Since $\widetilde{F}^{(i)^*}=\wV^{(i)^*}|_{\mathcal{W}_j},$ we have \begin{align}\label{KK7}
	\widetilde{F}^{(i)}\widetilde{F}^{(i)^*}&= P_{\mathcal{W}_j}\wV^{(i)}\wV^{(i)^*}|_{\mathcal{W}_j}=(\wV^{(i)}\wV^{(i)^*}-\wV^{(j)}\wV^{(j)^*}\wV^{(i)}\wV^{(i)^*})|_{\mathcal{W}_j}=I_{\mathcal{W}_j}-P_{\mathcal{W}_j}P_{\mathcal{W}_i}P_{\mathcal{W}_j}.
\end{align}

The following theorem establishes the relation between the Fringe operators and the {joint defect operator} of an isometric covariant representation, which is a generalization of \cite[Theorem 3.2]{HQY15}.
\begin{theorem}\label{H8}
	Let  $(\sigma , V^{(1)},V^{(2)})$  be  an isometric covariant representation of $\mathbb{E}$ on  $\mathcal{H}.$ Then $C(\sigma,V^{(1)}, V^{(2)})$  is compact if and only if both $I_{\mathcal{W}_j}-\widetilde{F}^{(i)}\widetilde{F}^{(i)^*}$ and $I_{E_i\ot \mathcal{W}_j}-\widetilde{F}^{(i)^*}\widetilde{F}^{(i)}$  are compact.
\end{theorem}
\begin{proof}
	Let  $1\le i,j \le 2$ with $i\neq j.$ Since $C$ is self-adjoint,   we rewrite $C$ as
	\begin{equation}\label{K8}
		C= P_{\mathcal{W}_j}P_{\mathcal{W}_i}-\wV^{(j)}[\wV^{(j)^*},\wV^{(i)}]\wV^{(i)^*}= P_{\mathcal{W}_i}P_{\mathcal{W}_j}-\wV^{(i)}[\wV^{(i)^*},\wV^{(j)}]\wV^{(j)^*}.
	\end{equation}
	Then, it is easy to verify that the operator $P_{\mathcal{W}_j}P_{\mathcal{W}_i}P_{\mathcal{W}_j}$ is orthogonal to $\wV^{(j)}[\wV^{(j)^*},\wV^{(i)}][\wV^{(i)^*},\wV^{(j)}]\\\wV^{(j)^*}$ and
	$$C^2=P_{\mathcal{W}_j}P_{\mathcal{W}_i}P_{\mathcal{W}_j}+ \wV^{(j)}[\wV^{(j)^*},\wV^{(i)}][\wV^{(i)^*},\wV^{(j)}]\wV^{(j)^*}.$$ 
	Observe that $C=0$ on $R(\wV),$ so we shall only consider $C$ on $\mathcal{W}\: (=R(\wV)^{\perp}).$ From Equation $(\ref{WWW})$  $\mathcal{W}=\mathcal{W}_{j} \oplus \widetilde{V}^{(j)}(E_j \otimes \mathcal{W}_{i}),$   then $C^2$ has of the form \begin{equation*}
		C^2 = 
		\begin{pmatrix}
			P_{\mathcal{W}_j}P_{\mathcal{W}_i}P_{\mathcal{W}_j} & 0 \\
			0 &  \wV^{(j)}[\wV^{(j)^*},\wV^{(i)}][\wV^{(i)^*},\wV^{(j)}]\wV^{(j)^*} \\
			
		\end{pmatrix}.
	\end{equation*}
	Since $\wV^{(j)}: E_j\ot \mathcal{W}_i \to \wV^{(j)}( E_j\ot \mathcal{W}_i)$ is a unitary,   $C^2$ is isomorphic to 
	\begin{equation}\label{W4}
		\begin{pmatrix}
			P_{\mathcal{W}_j}P_{\mathcal{W}_i}P_{\mathcal{W}_j} & 0 \\
			0 &  [\wV^{(j)^*},\wV^{(i)}][\wV^{(i)^*},\wV^{(j)}] \\
			
		\end{pmatrix}.
	\end{equation}
	It follows from the above observation and by using  Equations (\ref{K7}) and (\ref{KK7}), that the operators $I_{\mathcal{W}_j}-\widetilde{F}^{(i)}\widetilde{F}^{(i)^*}$ and $I_{E_i\ot \mathcal{W}_j}-\widetilde{F}^{(i)^*}\widetilde{F}^{(i)}$  are compact if and only if  $C^2$ is compact. Now $C$ is compact because  $C$ is self-adjoint.
\end{proof} 

Suppose that  $(\sigma , V^{(1)},V^{(2)})$ is a pure isometric covariant representation of $\mathbb{E}$ on $\mathcal{H}$ and  $(P_{\mathcal{W}_2},U,\mathcal{W})$ is the BCL-triple for $(\sigma , V^{(1)},V^{(2)}).$ We can rewrite the Fringe operators  $\widetilde{F}^{(i)}, 1 \leq i \leq 2,$ as 
$$\widetilde{F}^{(2)}=P_{\mathcal{W}_1}U(I_{E_2}\ot P_{\mathcal{W}_1})\quad and \quad (I_{E_2}\ot \widetilde{F}^{(1)})=( I_{E_2}\ot P_{\mathcal{W}_2})U^*(I_{E_2\ot E_1}\ot P_{\mathcal{W}_2}).$$
Note that  $\widetilde{F}^{(1)}$ and $\widetilde{F}^{(2)}$ are the compression of $U$and $U^*,$ respectively, so we will focus on $\widetilde{F}^{(2)}$ in the remaining part of this section.

\begin{theorem}\label{H6}
	Suppose that $(\sigma , V^{(1)},V^{(2)})$ is a pure isometric covariant representation of $\mathbb{E}$ on $\mathcal{H}$ and  $(P_{\mathcal{W}_2},U,\mathcal{W})$ is the BCL-triple  for $(\sigma , V^{(1)},V^{(2)}).$ The following conditions are equivalent:
	\begin{enumerate}
		\item  $C(\sigma,V^{(1)}, V^{(2)})\ge 0;$
		\item  $\mathcal{W}_1$ is an invariant subspace for $(\sigma,V^{(2)});$
		\item  $(\sigma , V^{(1)},V^{(2)})$ is doubly commuting;
		\item  $C(\sigma,V^{(1)}, V^{(2)})$ is a projection;
		\item   $(\sigma_2 , F^{(2)})$ is an isometric representation;
		\item   $U(E_2\ot \mathcal{W}_1)\subseteq \mathcal{W}_1.$
	\end{enumerate} 
\end{theorem}
\begin{proof}
	By Equations (\ref{L3}), (\ref{T2}) and Lemma \ref{T1}  we get $(1) \Leftrightarrow (2), (2) \Leftrightarrow (3), (1) \Leftrightarrow (6)$ and $(2) \Leftrightarrow (5).$ Now to prove $(3)$ implies $(4),$ we have 
	\begin{align*}
		C&
		=I-\wV^{(i)}\wV^{(i)^*}-\wV^{(j)}\wV^{(j)^*} +  \wV^{(i)}(I_{E_i} \ot \wV^{(j)})(t_{j,i}\ot I_{\mathcal{H}})(I_{E_j} \ot \wV^{(i)^*})\wV^{(j)^*}\\&=I-\wV^{(i)}\wV^{(i)^*}-\wV^{(j)}\wV^{(j)^*}+\wV^{(i)}\wV^{(i)^*}\wV^{(j)}\wV^{(j)^*}=(I-\wV^{(i)}\wV^{(i)^*})(I-\wV^{(j)}\wV^{(j)^*})= P_{\mathcal{W}_i}P_{\mathcal{W}_j},
	\end{align*} 
	for all $1 \leq i,j \leq 2$ with $i \neq j.$ Therefore $C$ is a projection. The part of $(4)$ implies $(1)$ is trivial. 
\end{proof}

\begin{theorem}\label{TT}
	Let $(\sigma , V^{(1)},V^{(2)})$ be an isometric covariant representation of $\mathbb{E}$ on $\mathcal{H}$ such that  $(\sigma,V^{(i)})$ is pure for some $1\le i \le 2$. Then $C(\sigma,V^{(1)}, V^{(2)})=0$ if and only if $C(\sigma,V^{(1)}, V^{(2)})\le 0.$
\end{theorem}

\begin{proof}
	Fix $1\le i \le 2$ and let $1\le j \le 2$ with $i\neq j$. Suppose that  $C\le 0,$ then using Equation (\ref{L3})  ${\mathcal{W}_j} \subseteq {\wV^{(i)} (E_i \ot \mathcal{W}_j)}$ and hence $${\mathcal{W}_j} \subseteq {\wV^{(i)}_n (E_i^{\ot n} \ot \mathcal{W}_j)} \subseteq {\wV^{(i)}_n (E_i^{\ot n} \ot \mathcal{H})},$$ for all $n \in \mathbb{N}.$  Since $(\sigma , V^{(i)})$ is pure,  we get  $${\mathcal{W}_j} \subseteq \bigcap_{n \in \mathbb{N}} {\wV^{(i)}_n (E_i^{\ot n} \ot \mathcal{W}_j)} \subseteq \bigcap_{n\in \mathbb{N}}{\wV^{(i)}_n (E_i^{\ot n} \ot \mathcal{H})}=\{0\}.$$
	This shows that  $C=P_{\mathcal{W}_j} - P_{\wV^{(i)} (E_i \ot \mathcal{W}_j)} = 0.$\end{proof} 

The following two propositions, prove a characterization of the joint defect operator of $(\sigma , V^{(1)},V^{(2)}),$ which are generalizations of \cite[Theorem 3.5]{HQY15} and \cite[Lemma 3.3 and  Lemma 4.6]{BRK22}. 

\begin{proposition}\label{U1}
	Suppose that $(\sigma , V^{(1)},V^{(2)})$ is an isometric covariant representation of $\mathbb{E}$ on $\mathcal{H}.$ Then the following conditions are equivalent:
	\begin{enumerate}
		\item $C(\sigma,V^{(1)}, V^{(2)})=0;$
		\item For $i \neq j,$ $\mathcal{W}_i$ is  $(\sigma , V^{(j)})$-reducing  and $(\sigma , V^{(j)})|_{\mathcal{W}_i}$ is co-isometric;
		\item  The covariant representations $(\sigma_1,F^{(1)})$ and $(\sigma_2,F^{(2)})$ are isometric and fully co-isometric;
		\item $\mathcal{W}=\mathcal{W}_1 \oplus \mathcal{W}_2;$
		\item $(R(\wV^{(1)}) \ominus R(\wV)) \oplus (R(\wV^{(2)}) \ominus R(\wV)) \oplus R(\wV)=\mathcal{H};$
		\item  If $(P_{\mathcal{W}_2},U,\mathcal{W})$ is the BCL-triple  for   $(\sigma , V^{(1)},V^{(2)}),$  then $U(E_2\ot \mathcal{W}_1) = \mathcal{W}_1.$
	\end{enumerate}
\end{proposition}
\begin{proof}
	The equivalences of $(1),(2),(3),(4)$ and $(6)$ easily follow from the  Equations $(\ref{WWW})$, (\ref{L3}) and Theorem \ref{H6}. Now to prove $(4)$ is equivalent to $(5)$, let $1\le i,j \le 2$ with $i\neq j.$
	
	$(4)\Rightarrow (5):$ Suppose  $\mathcal{W}=\mathcal{W}_1 \oplus \mathcal{W}_2,$  we  need to show that  $\mathcal{W}_j=R(\wV^{(i)}) \ominus R(\wV).$ Let $w\in \mathcal{W}_j,$ then  $w\in R(\wV^{(i)})$ and $w \in R(\wV)^{\perp},$ that is,  $w\in R(\wV^{(i)}) \ominus R(\wV).$ If $w\in R(\wV^{(i)}) \ominus R(\wV),$ then $w\in R(\wV^{(i)})$ and $w\in \mathcal{W}.$ So $w\in \mathcal{W}_j$ and hence (5) holds.  
	
	$(5)\Rightarrow (4):$ Suppose $(5)$ holds. Since $R(\wV)\subseteq R(\wV^{(j)}),$ we have $R(\wV^{(j)})=(R(\wV^{(j)}) \ominus R(\wV)) \oplus R(\wV).$ Therefore $\mathcal{W}_j=R(\wV^{(j)})^{\perp}=R(\wV^{(i)}) \ominus R(\wV)$ and hence $\mathcal{W}_1 \oplus \mathcal{W}_2=\mathcal{W}.$
\end{proof}

\begin{proposition}
	Let $(\sigma , V^{(1)},V^{(2)})$ be an isometric covariant representation of $\mathbb{E}$ on $\mathcal{H}.$  The following conditions are equivalent:
	\begin{enumerate}
		\item $C(\sigma,V^{(1)}, V^{(2)})< 0;$ 
		\item  $\mathcal{W}_i \subsetneq {\wV^{(j)} (E_j \ot \mathcal{W}_i)};$
		\item Both $\widetilde{F}^{(1)}$ and $\widetilde{F}^{(2)}$ are co-isometries but  not unitaries;
		\item $\mathcal{W}_1$ is orthogonal to $\mathcal{W}_2$ and  $\mathcal{W} \neq \mathcal{W}_1 \oplus \mathcal{W}_2;$
		\item $C(\sigma,V^{(1)}, V^{(2)})$ is the negative of a non-zero projection;
		\item   If $(P_{\mathcal{W}_2},U,\mathcal{W})$ is the BCL-triple  for   $(\sigma , V^{(1)},V^{(2)}),$  then  $\mathcal{W}_1 \subsetneq U(E_2\ot \mathcal{W}_1).$
	\end{enumerate}
\end{proposition}
\begin{proof}
	The proof follows from Equation $(\ref{L3})$ and Lemma \ref{U1}. 	
\end{proof}

Since $C(\sigma,V^{(1)}, V^{(2)})$ is self-adjoint. If $C$ is compact, its nonzero eigenvalues are in $[-1,1].$ Let $T$ be a bounded linear operator on $\mathcal{H}$ and $\mathcal{E}_{\lambda}(T)$ denotes the eigenspace corresponding to the eigenvalue $\lambda$ of $T$. 

The following theorem is a generalization of \cite[Proposition 4.1]{HQY15}:
\begin{theorem}\label{H5}
	Let $C(\sigma,V^{(1)}, V^{(2)})$ be the joint defect operator of an isometric covariant representation $(\sigma ,V^{(1)},\\V^{(2)})$ of $\mathbb{E}$ on $\mathcal{H}$. Then
	\begin{enumerate}
		\item $\mathcal{E}_1(C)=\mathcal{H} \ominus (\wV^{(1)}(E_1 \ot \mathcal{H}) + \wV^{(2)}(E_2 \ot \mathcal{H}));$  
		\item $\mathcal{E}_{-1}(C)= (\wV^{(1)}(E_1 \ot \mathcal{H}) \cap \wV^{(2)}(E_2 \ot \mathcal{H})) \ominus \wV(E\ot \mathcal{H}).$
	\end{enumerate}
\end{theorem}
\begin{proof}
	
	$(1)$ Let	 $h \in \mathcal{H} \ominus (\wV^{(1)}(E_1 \ot \mathcal{H}) + \wV^{(2)}(E_2 \ot \mathcal{H}))= N(\wV^{(1)^*}) \cap N(\wV^{(2)^*}),$ then using Equation (\ref{WWW}) we have 
	\begin{align*}
		(C-I)h= \wV\wV^*h -\wV^{(1)}\wV^{(1)^*}h-\wV^{(2)}\wV^{(2)^*}h=0.
	\end{align*} That is,   $\mathcal{H} \ominus (\wV^{(1)}(E_1 \ot \mathcal{H}) + \wV^{(2)}(E_2 \ot \mathcal{H})) \subseteq \mathcal{E}_1(C).$ Note that $	I-C= \wV^{(1)}\wV^{(1)^*}+ \wV^{(2)}(I_{E_2}\ot P_{\mathcal{W}_1})\wV^{(2)^*},$ for every $h \in \mathcal{H}$  \begin{align}\label{CCC}
		\langle(I-C)h,h\rangle=\|\wV^{(1)^*}h\|^2+\|(I_{E_2}\ot P_{\mathcal{W}_1})\wV^{(2)^*}h\|^2. \end{align}
	This shows that, for all $h \in \mathcal{E}_1(C),$ $(I-C)h=0$ if and only if   $\wV^{(1)^*}h=0$
	and $(I_{E_2}\ot P_{\mathcal{W}_1})\wV^{(2)^*}h=0.$ Also,  it gives 
	\begin{align*}
		0=(I_{E_2}\ot P_{\mathcal{W}_1})\wV^{(2)^*}h&=\wV^{(2)^*}h-(I_{E_2}\ot \wV^{(1)})(t_{1,2}\ot I_{\mathcal{H}})(I_{E_1}\ot \wV^{(2)^*})\wV^{(1)^*}h=\wV^{(2)^*}h.
	\end{align*}Therefore  $ \mathcal{E}_1(C)=N(\wV^{(1)^*}) \cap N (\wV^{(2)^*})=\mathcal{H} \ominus (\wV^{(1)}(E_1 \ot \mathcal{H}) + \wV^{(2)}(E_2 \ot \mathcal{H})).$
	
	$(2)$ Let $h \in (\wV^{(1)}(E_1 \ot \mathcal{H}) \cap \wV^{(2)}(E_2 \ot \mathcal{H})) \ominus \wV(E\ot \mathcal{H}),$ then there exist $\eta_1\in E_1\ot \mathcal{H}$ and $\eta_2\in E_2\ot \mathcal{H}$ such that $h=\wV^{(1)}(\eta_1)=\wV^{(2)}(\eta_2)$ and $h \in N(\wV^*).$ Then  
	\begin{align*}
		(C+I)h=2h-\wV^{(1)}\wV^{(1)^*}h-\wV^{(2)}\wV^{(2)^*}h=2h-\wV^{(1)}(\eta_1)-\wV^{(2)}(\eta_2)=0.
	\end{align*} That is,  $(\wV^{(1)}(E_1 \ot \mathcal{H}) \cap \wV^{(2)}(E_2 \ot \mathcal{H})) \ominus \wV(E\ot \mathcal{H}) \subseteq \mathcal{E}_{-1}(C).$  On the other hand, let $h\in \mathcal{E}_{-1}(C),$ then we have  
	\begin{align*}
		0=\langle (C+I)h,h \rangle &=\langle P_{\mathcal{W}_1}h,h \rangle +\langle P_{\mathcal{W}_2}h,h \rangle +\langle \wV\wV^*h,h \rangle =\|P_{\mathcal{W}_1}h\|^2 + \|P_{\mathcal{W}_2}h\|^2 + \|\wV^*h\|^2.
	\end{align*} Therefore $P_{\mathcal{W}_1}h=0,P_{\mathcal{W}_2}h=0$ and $h\in N (\wV^*)$ and hence 
	$h \in (\wV^{(1)}(E_1 \ot \mathcal{H}) \cap \wV^{(2)}(E_2 \ot \mathcal{H})) \ominus \wV(E\ot \mathcal{H}).\qedhere$
	
\end{proof}

\begin{theorem}
	Suppose that $(\sigma , V^{(1)},V^{(2)})$ is an isometric covariant representation of $\mathbb{E}$ on  $\mathcal{H}.$ Let $\lambda$ be a nonzero eigenvalue of $C(\sigma,V^{(1)}, V^{(2)})$ in $(-1,1).$ Then $-\lambda$ is an eigenvalue of $C(\sigma,V^{(1)}, V^{(2)})$ and dim $\mathcal{E}_{\lambda}(C)$ = dim $ \mathcal{E}_{-\lambda}(C).$
\end{theorem}
\begin{proof}
	Let $\lambda$ be a nonzero eigenvalue of  $C,$ then there exists a non-zero element $h\in \mathcal{E}_{\lambda}(C)$ such that $C h= \lambda h.$ Therefore by (\ref{CCC}),  $\wV^{(i)^*}h\neq 0$ for $1\le i\le 2.$ Now we get \begin{align*} 
		\lambda\wV^{(1)^*}h&=\wV^{(1)^*}(I-\wV^{(1)}\wV^{(1)^*}-\wV^{(2)}\wV^{(2)^*}+\wV\wV^*)h\\&= (-\wV^{(1)^*}\wV^{(2)}+(I_{E_1}\ot \wV^{(2)})(t_{2,1}\ot I_{\mathcal{H}})(I_{E_1}\ot \wV^{(1)^*}))\wV^{(2)^*}h=-[\wV^{(1)^*},\wV^{(2)}]\wV^{(2)^*}h.
	\end{align*}Similarly, we obtain 
	$\lambda\wV^{(2)^*}h=-[\wV^{(2)^*},\wV^{(1)}]\wV^{(1)^*}h.$ Therefore $$\lambda^2\wV^{(1)^*}h=[\wV^{(1)^*},\wV^{(2)}][\wV^{(2)^*},\wV^{(1)}]\wV^{(1)^*}h.$$ It follows that, the  restriction map $$\wV^{(1)^*}|_{ \mathcal{E}_{\lambda}(C)}: \mathcal{E}_{\lambda}(C) \to \mathcal{E}_{\lambda^2}([\wV^{(1)^*},\wV^{(2)}][\wV^{(2)^*},\wV^{(1)}]),$$ $h\mapsto \wV^{(1)^*}h$ is one-to-one and hence
	\begin{equation}\label{W2}
		dim \mathcal{E}_{\lambda}(C) \le dim \mathcal{E}_{\lambda^2}([\wV^{(1)^*},\wV^{(2)}][\wV^{(2)^*},\wV^{(1)}]).
	\end{equation} For $1\le i,j\le 2$ with $i\neq j,$ we have  
	$\lambda P_{\mathcal{W}_j}h=P_{\mathcal{W}_j}Ch=P_{\mathcal{W}_j}P_{\mathcal{W}_i}h, \: h\in \mathcal{E}_{\lambda}(C)$ and thus  $$\lambda^2 P_{\mathcal{W}_1}h=P_{\mathcal{W}_1}P_{\mathcal{W}_2}P_{\mathcal{W}_1}h.$$ Suppose that $P_{\mathcal{W}_i}h= 0,$ for some non zero $ h \in   \mathcal{E}_{\lambda}(C)$ and  $\lambda \in (-1, 1), $ then $P_{\mathcal{W}_j}h= 0.$ Note that  $C=0 $ on $R(\wV),$ and $h \in ( R(\wV^{(1)}) \cap R(\wV^{(2)})) \ominus R(\wV)=\mathcal{E}_{-1}(C)$ (by Theorem \ref{H5}), which is a contradiction. Therefore  $P_{\mathcal{W}_i}h\neq 0,$ $1\le i \le 2,$ and hence the restriction  map $P_{\mathcal{W}_1}|_{ \mathcal{E}_{\lambda}(C) }: \mathcal{E}_{\lambda}(C) \to \mathcal{E}_{\lambda^2}(P_{\mathcal{W}_1}P_{\mathcal{W}_2}P_{\mathcal{W}_1})$ by $h\mapsto P_{\mathcal{W}_1}h$ is  one-to-one. This  gives \begin{equation}\label{W3}
		dim \mathcal{E}_{\lambda}(C) \le dim \mathcal{E}_{\lambda^2}(P_{\mathcal{W}_1}P_{\mathcal{W}_2}P_{\mathcal{W}_1}).
	\end{equation} From Equations (\ref{W2}),(\ref{W3}) and  (\ref{W4}), we get \begin{align*}
		dim \mathcal{E}_{\lambda^2}(C^2)\ge \mbox{2} dim \mathcal{E}_{\lambda}(C). 
	\end{align*} If $\lambda^2$ is an eigenvalue of $C^2,$  then either $\lambda$ or $-\lambda$ is an eigenvalue of $C$. Hence,
	$dim \mathcal{E}_{\lambda^2}(C^2)= dim \mathcal{E}_{\lambda}(C) +dim \mathcal{E}_{-\lambda}(C).$ Therefore $dim \mathcal{E}_{\lambda}(C) \le dim \mathcal{E}_{-\lambda}(C).$ Similarly, using the same  approach for $-\lambda,$ we get  $dim \mathcal{E}_{\lambda}(C) \ge dim \mathcal{E}_{-\lambda}(C).$ This completes the proof.
\end{proof}

Let $(\sigma , V^{(1)},V^{(2)})$ be an isometric covariant representation of $\mathbb{E}$ on  $\mathcal{H}.$ Suppose that the joint defect operator $C$ is compact, then ${N(C)^{\perp}}$ can be written as
\begin{equation*}
	{N(C)^{\perp}}=  \mathcal{E}_{1}(C) \oplus \big(\bigoplus_{0 < \lambda_i < 1} \mathcal{E}_{\lambda_i}(C)\big) \oplus \mathcal{E}_{-1}(C) \oplus \big(\bigoplus_{-1 < \lambda_i < 0} \mathcal{E}_{\lambda_i}(C)\big).
\end{equation*}Let $d_1=dim \mathcal{E}_{1}(C), d_2= dim \mathcal{E}_{-1}(C)$ and $B= \bigoplus_{0< \lambda_i< 1} \lambda_i P_i,$ where $P_i: \mathcal{H}\to \mathcal{E}_{\lambda_i}(C)$ be the orthogonal projection onto $\mathcal{E}_{\lambda_i}(C).$  Then we have the following theorem:

\begin{theorem}\label{A1}
	Suppose that $(\sigma , V^{(1)},V^{(2)})$ is an isometric covariant representation of $\mathbb{E}$ on  $\mathcal{H}$ such that $C(\sigma,V^{(1)},\\V^{(2)})$ is compact. Then $C(\sigma,V^{(1)},V^{(2)})|_{{N(C)^{\perp}}}$ is isomorphic to  
	\begin{equation*}
		\begin{pmatrix}
			I_{d_1} & 0 & 0 & 0 \\
			0 &  B & 0 & 0 \\
			0 & 0 & -I_{d_2} & 0 \\
			0 & 0 & 0 & -B \\
		\end{pmatrix},
	\end{equation*} where $I_{n}$ is the $n \times n$ identity matrix.
\end{theorem}

\begin{corollary}
	Let $(\sigma , V^{(1)},V^{(2)})$ be an isometric covariant representation of $\mathbb{E}$ on  $\mathcal{H}.$ Suppose that $C(\sigma,V^{(1)}, V^{(2)})$ is compact and negative (or positive), then rank of $C(\sigma,V^{(1)}, V^{(2)})$ is finite.
\end{corollary}

\section{Fredholm index for covariant representation of a product system over $\mathbb{N}_0^{2}$}

In this section, we introduce the Fredholm index of a covariant representation of the product system $\mathbb{E}$ and prove the index of the Fringe operator is the same as the Fredholm index of an isometric covariant representation of a product system over $\mathbb{N}_0^2.$ Also it's related to the eigenspaces of compact joint defect operator $C(\sigma,V^{(1)}, V^{(2)}).$ One of the main results of this section, Corollary \ref{CRR}, is a classification result for the isometric covariant representations of the product system over $\mathbb{N}_0^2$ using the congruence relation discussed in \cite{HQY15,Y05}

A covariant representation $(\sigma, V)$ of $E$ on $\mathcal{H}$ is said to be  {\it Fredholm} if $R(\wV)$ is closed, and both $N(\wV)$ and $N(\wV^*)$ have finite dimensions.  The index of $(\sigma, V)$ in this case is defined by  $$ind(\sigma,V) = dim(N(\wV)) - dim(N(\wV^*)).$$
Let  $(\sigma , V^{(1)},V^{(2)})$ be  a covariant representation of $\mathbb{E}$ on  $\mathcal{H},$ then for $E:=E_1\ot E_2,$ there is a short sequence 
\[
\begin{tikzcd}
	0 \arrow{r} &	E \ot \mathcal{H} \arrow{r}{d_1}  & (E_1 \ot \mathcal{H}) \oplus (E_2\ot \mathcal{H}) \arrow{r}{d_2} 
	& \mathcal{H} \arrow{r} &0 ,
\end{tikzcd}
\]
where $$d_1(\eta)= (-(I_{E_1} \ot \wV^{(2)})\eta,(I_{E_2} \ot \wV^{(1)})(t_{1,2}\ot I_{\mathcal{H}})\eta)\quad \: \mbox{and} \: \quad  d_2(\eta_1,\eta_2)=\wV^{(1)} \eta_1 +\wV^{(2)}\eta_2,$$ for  $\eta \in E \ot \mathcal{H}$ and $\eta_i\in E_i \ot \mathcal{H}, 1\le i\le 2.$
It is easy to see that  $d_2d_1=0.$ Indeed, let $\eta \in E \ot \mathcal{H},$ then we have 
\begin{align*}
	d_2d_1\eta&=d_2(-(I_{E_1} \ot \wV^{(2)})\eta,(I_{E_2} \ot \wV^{(1)})(t_{1,2}\ot I_{\mathcal{H}})\eta)\\&=-\wV^{(1)}(I_{E_1} \ot \wV^{(2)})\eta+ \wV^{(2)}(I_{E_2} \ot \wV^{(1)})(t_{1,2}\ot I_{\mathcal{H}})\eta=0.
\end{align*}

\begin{definition}
	A covariant representation $(\sigma , V^{(1)},V^{(2)})$  of $\mathbb{E}$ on  $\mathcal{H}$ is said to be (jointly) {\rm Fredholm} if both $d_1$ and $d_2$ have closed range  and $$dim(N(d_1))+dim[N (d_2)\ominus R(d_1)]+dim[\mathcal{H}\ominus R(d_2)] < \infty.$$ The {\rm index} of $(\sigma , V^{(1)},V^{(2)})$ is defined by $$ind(\sigma , V^{(1)},V^{(2)})=-dim(N(d_1))+dim[N(d_2)\ominus R(d_1)]- dim[\mathcal{H}\ominus R(d_2)].$$
\end{definition}

Suppose that $(\sigma , V^{(1)},V^{(2)})$ is an isometric covariant representation of $\mathbb{E}$ on  $\mathcal{H},$ then
it is easy to see that $N(d_1)=0.$ The following lemma is helpful to derive the proof of Theorem \ref{AAA}.
\begin{lemma}\label{K4}
	Let $(\sigma , V^{(1)},V^{(2)})$ be an isometric covariant representation of $\mathbb{E}$ on  $\mathcal{H}.$ Then
	\begin{equation*}\label{H3}
		N(d_2) \ominus R(d_1) = \{(\eta_1,-\wV^{(2)^*}\wV^{(1)}\eta_1) : \eta_1 \in R(I_{E_1} \ot \wV^{(2)})^{\perp} , \wV^{(1)}\eta_1 \in R(\wV^{(2)})\}.
	\end{equation*}
\end{lemma}
\begin{proof}
	Let $y=(\eta_1,-\wV^{(2)^*}\wV^{(1)}\eta_1)$ in the R.H.S. of the above Equation, then we have
	\begin{align*}
		d_2(\eta_1,-\wV^{(2)^*}\wV^{(1)}\eta_1)&=\wV^{(1)}\eta_1 - \wV^{(2)}\wV^{(2)^*}\wV^{(1)}\eta_1=\wV^{(1)}\eta_1-\wV^{(1)}\eta_1=0.
	\end{align*}For each $\eta\in E \ot \mathcal{H},$ we get 
	\begin{align*}
		\langle (\eta_1,-\wV^{(2)^*}\wV^{(1)}\eta_1), d_1(\eta) \rangle &= -\langle \eta_1,(I_{E_1} \ot \wV^{(2)})\eta \rangle - \langle \wV^{(2)^*}\wV^{(1)}\eta_1, (I_{E_2} \ot \wV^{(1)})(t_{1,2}\ot I_{\mathcal{H}})\eta \rangle \\&=- \langle \eta_1,(I_{E_1} \ot \wV^{(2)})\eta \rangle - \langle \wV^{(1)}\eta_1, \wV^{(1)}(I_{E_1} \ot \wV^{(2)})\eta \rangle \\&= -2 \langle \eta_1,(I_{E_1} \ot \wV^{(2)})\eta \rangle =0.
	\end{align*} On the other hand, let $(\eta_1,\eta_2)\in  N(d_2) \ominus R(d_1),$  then 
	$\wV^{(1)}\eta_1 + \wV^{(2)}\eta_2=d_2(\eta_1,\eta_2)=0$ and thus  $\eta_2=-\wV^{(2)^*}\wV^{(1)}\eta_1$ and $\wV^{(1)}\eta_1 \in R(\wV^{(2)}).$ Since $(\eta_1,\eta_2)\in R(d_1)^{\perp},$  $\eta_1 \in R(I_{E_1} \ot \wV^{(2)})^{\perp}.$ 
\end{proof}

The following lemma establishes the range relation between the Fringe operator and the isometric covariant representation of the product system $\mathbb{E}$ over $\mathbb{N}_0^2$, which is an analogue of \cite[Lemma 1.1]{Y03}.

\begin{lemma}\label{K5}
	Let $(\sigma , V^{(1)},V^{(2)})$ be an isometric covariant representation of $\mathbb{E}$ on  $\mathcal{H}.$ Then $$R(\widetilde{F}^{(2)})=(R(\wV^{(1)})+R(\wV^{(2)})) \ominus R(\wV^{(1)}).$$ Moreover, $R(\widetilde{F}^{(2)})$ is closed if and only if $R(\wV^{(1)})+R(\wV^{(2)})$ is closed.  Also, $$N(\widetilde{F}^{(2)^*})=\mathcal{H} \ominus (R(\wV^{(1)})+R(\wV^{(2)})).$$
\end{lemma}
\begin{proof}
	Let $\eta_2\in E_2\ot \mathcal{W}_1,$   $\widetilde{F}^{(2)}\eta_2=P_{\mathcal{W}_1}\wV^{(2)}\eta_2=\wV^{(2)}\eta_2-\wV^{(1)}\wV^{(1)^*}\wV^{(2)}\eta_2.$ That is, $$R(\widetilde{F}^{(2)}) \subseteq (R(\wV^{(1)})+R(\wV^{(2)})) \ominus R(\wV^{(1)}).$$
	On the other hand,   let  $h \in (R(\wV^{(1)})+R(\wV^{(2)})) \ominus R(\wV^{(1)}),$ then there exist $\eta_1\in E_1\ot \mathcal{H}$ and $\eta_2\in E_2\ot \mathcal{H}$  such that $h=\wV^{(1)}\eta_1+\wV^{(2)}\eta_2.$ Since $h \in \mathcal{W}_1$ and  for each $\eta'_1 \in E_1 \ot \mathcal{H},$ we get
	\begin{align*}
		\langle \eta_1 +\wV^{(1)^*}\wV^{(2)}\eta_2,\eta'_1\rangle &= \langle \eta_1,\eta'_1\rangle + \langle \wV^{(1)^*}\wV^{(2)}\eta_2,\eta'_1\rangle =\langle \wV^{(1)}\eta_1+\wV^{(2)}\eta_2,\wV^{(1)}\eta_1'\rangle =0.
	\end{align*} It follows that $\eta_1 +\wV^{(1)^*}\wV^{(2)}\eta_2=0$ and hence
	$h=\wV^{(1)}\eta_1+\wV^{(2)}\eta_2=(I-\wV^{(1)}\wV^{(1)^*})\wV^{(2)}\eta_2.$ Since $\eta_2=(I_{E_2}\ot (I-\wV^{(1)}\wV^{(1)^*}))\eta_2+(I_{E_2}\ot \wV^{(1)}\wV^{(1)^*})\eta_2$ and using the commutant relation of $(\sigma , V^{(1)},V^{(2)}),$ we have 
	\begin{align*}
		h&=(I-\wV^{(1)}\wV^{(1)^*})\wV^{(2)}(I_{E_2}\ot (I-\wV^{(1)}\wV^{(1)^*}))\eta_2= P_{\mathcal{W}_1}\wV^{(2)}|_{E_2\ot {\mathcal{W}_1}}\eta_2=\widetilde{F}_2(\eta_2).
	\end{align*} Therefore $(R(\wV^{(1)})+R(\wV^{(2)})) \ominus R(\wV^{(1)})\subseteq R(\widetilde{F}^{(2)}).$
\end{proof}

By the definition of Fringe operator, it is easy to see that
\begin{align}\label{K6}
	N(\widetilde{F}^{(2)})=\{\eta_2\in E_2\ot \mathcal{W}_1: \wV^{(2)}\eta_2 \in R(\wV^{(1)})\}.
\end{align}

From Lemma \ref{K4} and Equation $(\ref{K6}),$ define a bijection map from    $N(d_2) \ominus R(d_1)$ to $N(\widetilde{F}^{(2)})$ by $ (\eta,-\wV^{(2)^*}\wV^{(1)}\eta) \mapsto \wV^{(2)^*}\wV^{(1)}\eta$ and thus  dim$(N(\widetilde{F}^{(2)}))$=dim$(N(d_2) \ominus R(d_1)).$ This gives the following theorem which is a generalization of \cite[Theorem 1.3]{Y03}.
\begin{theorem}\label{AAA}
	Suppose that $(\sigma , V^{(1)},V^{(2)})$ is an isometric covariant representation of $\mathbb{E}$ on  $\mathcal{H}.$ Then $(\sigma_2,F^{(2)})$ is Fredholm if and only if $(\sigma , V^{(1)},V^{(2)})$ is Fredholm. Moreover, $$ind(\sigma_2,F^{(2)})=ind(\sigma , V^{(1)},V^{(2)}).$$
\end{theorem}

For the  joint defect operator $C$   of  $(\sigma , V^{(1)},V^{(2)})$   and by  using Theorem \ref{H5},  the map $\eta_2 \mapsto \wV^{(2)}\eta_2$   from $N(\widetilde{F}^{(2)})$ to $\mathcal{E}_{-1}(C)$ is well defined bijection. The following theorem follows from the above, and using Theorems \ref{H8} and  \ref{H5}.

\begin{theorem}\label{EEE}
	Let  $(\sigma , V^{(1)},V^{(2)})$  be an isometric covariant representation of $\mathbb{E}$ on  $\mathcal{H}.$ If  the joint defect operator $C$ is compact, then  $(\sigma_2,F^{(2)})$ is Fredholm and  $$ind(\sigma_2,F^{(2)})= dim \mathcal{E}_{-1}(C) - dim \mathcal{E}_{1}(C).$$
\end{theorem}

Let $\mathcal{K}$ be an invariant subspace of an isometric covariant representation $(\sigma , V^{(1)},V^{(2)})$ of $\mathbb{E}$ on $\mathcal{H}.$ We  denote the joint defect operator $C(\sigma,V^{(1)}, V^{(2)})$ restricted to  $\mathcal{K}$ by $C^{\mathcal{K}}(\sigma',V^{(1)}, V^{(2)}),$ where $\sigma'=\sigma|_{\mathcal{K}}.$ That is,
$$C^{\mathcal{K}}(\sigma',V^{(1)}, V^{(2)})=I-\wV^{(1)}(\wV^{(1)}|_{E_1\ot\mathcal{K}})^*-\wV^{(2)}(\wV^{(2)}|_{E_2\ot\mathcal{K}})^*-\wV(\wV|_{E\ot\mathcal{K}})^*.$$

\begin{definition}
	Let $(\sigma , V^{(1)},V^{(2)})$ be an isometric covariant representation of $\mathbb{E}$ on  $\mathcal{H}.$  The two invariant subspaces ${\mathcal{K}_1}$ and ${\mathcal{K}_2}$ of  $(\sigma , V^{(1)},V^{(2)})$ are said to be {\rm congruent} if there exists an invertible module map $L:{\mathcal{K}_1}\to {\mathcal{K}_2}$  (i.e., $L(ah)=aL(h)$ and $L$ is invertible, $a \in \mathcal{A}, h \in \mathcal{K}$) such that $$C^{\mathcal{K}_2}=LC^{\mathcal{K}_1}L^*.$$
\end{definition}

Suppose that the joint defect operator $C$  is finite rank, then it can be written as a symmetric invertible matrix (see Theorem \ref{A1}) on ${N(C)^{\perp}}$. An invertible  symmetric matrix $T$
is said to have \emph{signature} $(p, q)$ if there is an invertible matrix $B$ such that $BTB^*= \begin{pmatrix}
	I_{p} & 0  \\
	0 & - I_{q} \\
\end{pmatrix}.$

The following theorem is a generalization of \cite[Proposition 3.1]{Y05}.
\begin{theorem}
	Let $\mathcal{K}_1$ and $\mathcal{K}_2$   be   $(\sigma , V^{(1)},V^{(2)})$-invariant subspaces of  $\mathcal{H}$  such that the dim$(\mathcal{K}_1)$ =dim$(\mathcal{K}_2)$ and let $C^{\mathcal{K}_1}$ has finite rank. Then  $\mathcal{K}_2$    is congruent to $\mathcal{K}_1$ if and only if both $C^{\mathcal{K}_1}$ and $C^{\mathcal{K}_2}$ have the same signature.
\end{theorem}
\begin{proof}
	We can write  the subspaces ${\mathcal{K}_1}$ and ${\mathcal{K}_2}$ as $$\mathcal{K}_1=N (C^{\mathcal{K}_1})\oplus N(C^{\mathcal{K}_1})^{\perp}\quad and \quad \mathcal{K}_2= N(C^{\mathcal{K}_2})\oplus N (C^{\mathcal{K}_2})^{\perp}.$$ Suppose that $\mathcal{K}_1$  is congruent to $\mathcal{K}_2,$ then there exists an invertible module map $L:{\mathcal{K}_1}\to {\mathcal{K}_2}$ such that $C^{\mathcal{K}_2}=LC^{\mathcal{K}_1}L^*.$ It is easy to verify that $L^*$ maps $N(C^{\mathcal{K}_2})$ onto $N(C^{\mathcal{K}_1})$ and  $C^{\mathcal{K}_2}=LC^{\mathcal{K}_1}L^*$ on $ N(C^{\mathcal{K}_2})^{\perp}.$ Therefore  both $C^{\mathcal{K}_1}$ and $C^{\mathcal{K}_2}$ have the same signature. 
	
	Conversely, let  $C^{\mathcal{K}_1}$ and $C^{\mathcal{K}_2}$ have the same signature. That means, the joint defect operators restricted to the orthogonal complement of their kernels have the same signature. Then there exists an invertible module map $L_1$ from $N(C^{\mathcal{K}_1})^{\perp}$ onto $ N(C^{\mathcal{K}_2})^{\perp}$ such that $C^{\mathcal{K}_2}|_{N(C^{\mathcal{K}_2})^{\perp}}=L_1C^{\mathcal{K}_1}L_1^*.$ Since  dim$(N(C^{\mathcal{K}_1})^{\perp})$=dim$ (N(C^{\mathcal{K}_2})^{\perp}) < \infty,$  choose
	$L_2$  to  be any  invertible module map from $N(C^{\mathcal{K}_1})$ onto $N(C^{\mathcal{K}_2}).$  Consider  $L=L_1\oplus L_2,$ we get $C^{\mathcal{K}_2}=LC^{\mathcal{K}_1}L^*.$
\end{proof}

\begin{corollary}\label{CRR}
	Let $(\sigma , V^{(1)},V^{(2)})$ and $(\sigma' , V^{(1)'},V^{(2)'})$ be an isometric covariant representations of $\mathbb{E}$ on the Hilbert spaces  $\mathcal{H}$ and $\mathcal{H}',$ respectively such that  dim$(\mathcal{H})$=dim$(\mathcal{H}')$ and let $C(\sigma,V^{(1)},V^{(2)})$ has finite rank. Then $(\sigma,V^{(1)},V^{(2)})$ is congruent to $(\sigma' , V^{(1)'},V^{(2)'})$ (that is, there is an invertible module map $L:\mathcal{H}\to \mathcal{H}'$ such that $C(\sigma' , V^{(1)'},V^{(2)'})=LC(\sigma,V^{(1)},V^{(2)})L^*$ and $L\sigma(a)=\sigma'(a)L$ for all $a\in\mathcal{A}$) if and only if both $C(\sigma,V^{(1)},V^{(2)})$ and $C(\sigma' , V^{(1)'},V^{(2)'})$ have the same signature.
\end{corollary}
The following corollary follows from the above corollary and Theorems \ref{AAA} and \ref{EEE}.
\begin{corollary}
	Let $(\sigma , V^{(1)},V^{(2)})$ and $(\sigma' , V^{(1)'},V^{(2)'})$ be an isometric covariant representations of $\mathbb{E}$ on  $\mathcal{H}$ and $\mathcal{H}',$ respectively. Suppose that dim$(\mathcal{H})$=dim$(\mathcal{H}')$ and both $C(\sigma,V^{(1)},V^{(2)})$ and $C(\sigma' , V^{(1)'},V^{(2)'})$ are finite rank and have the same signature. Then $$ind(\sigma' , V^{(1)'},V^{(2)'})=ind(\sigma , V^{(1)},V^{(2)}).$$
\end{corollary}

\paragraph{$\mathbf{Acknowledgement}$} Dimple Saini is supported by a UGC fellowship (File No:16-6(DEC. 2018)/2019(NET/CSIR)). Harsh Trivedi is supported by MATRICS-SERB  
Research Grant, File No: MTR/2021/000286, by 
SERB, Department of Science \& Technology (DST), Government of India. Shankar Veerabathiran thanks ISI Bangalore for Visiting Scientist position. We acknowledge the DST-FIST program (Govt. of India) for providing financial support for setting up the computing lab facility under the scheme “Fund for Improvement of Science and Technology” (FIST - No. SR/FST/MS-I/2018/24).


\begin{thebibliography}{10}
	\bibitem{A89}
	W. Arveson, \textit{Continuous analogues of {F}ock space}, Mem. Amer. Math. Soc. \textbf{80} (1989), no.~409, iv+66.
	
	\bibitem{BDF06}
	H. Bercovici, R.G. Douglas and C. Foias, \emph{On the classification of multi-isometries}, Acta Sci. Math. (Szeged), \textbf{72} (2006), 639-661.
	
	\bibitem{BCL78}
	C.A. Berger, L. A. Coburn and A. Lebow, \emph{Representation and index theory for $C\sp*$-algebras generated by commuting isometries},
	J. Functional Analysis \textbf{27} (1978), no. 1,
	51-99.
	
	\bibitem{BRK22}
	T. Bhattacharyya, S. Rastogi and U. V. Kumar,  \emph{The Joint Spectrum for a Commuting Pair of Isometries in Certain Cases}, Complex Anal. Oper. Theory \textbf{16}, 83 (2022).
	
	\bibitem{C77}
	J. Cuntz, \textit{Simple $C^*$-algebras generated by isometries}, Comm. Math. Phys.
	\textbf{57}, (1977), no.~2, 173-185.
	
	\bibitem{C81}
	R. E. Curto,  \emph{Fredholm and invertible {$n$}-tuples of operators. {T}he
		deformation problem}, Trans. Amer. Math. Soc.
	\textbf{266} (1981), no.~1, 129-159.
	
	
	\bibitem{DSSS22}
	S. De, P. Shankar, J. Sarkar and Sankar T.R, \emph{Pairs of projections and commuting isometries}, to appear in J. Operator Theory.
	
	
	\bibitem{F02}
	N.J. Fowler, \emph{Discrete product systems of Hilbert bimodules}, Pacific J. Math.
	\textbf{204} (2002), no.~2, 335-375.
	
	
	\bibitem{HQY15}
	W. He, Y. Qin and R. Yang, \emph{Numerical invariants for commuting isometric pairs}, Indiana Univ. Math. J. \textbf{64} (2015), 1-19.
	
	
	\bibitem{L95}
	E.~C. Lance, \emph{Hilbert {$C^*$}-modules}, London Mathematical Society
	Lecture Note Series, vol. \textbf{210}, Cambridge University Press, Cambridge, 1995, A
	toolkit for operator algebraists.
	
	\bibitem{MMSS19}
	A. Maji, A. Mundayadan, J. Sarkar and Sankar T. R, \emph{Characterization of invariant subspaces in the polydisc}, J. Operator Theory \textbf{82} (2019), no. 2, 445-468.
	
	\bibitem{MSS18}
	A. Maji, J. Sarkar and Sankar T. R, \emph{Pairs of Commuting Isometries-I}, Studia Math. \textbf{248} (2019), no. 2, 171-189. 
	
	\bibitem{MS98}
	P.~S. Muhly and B. Solel, \emph{Tensor algebras over
		{$C^*$}-correspondences: representations, dilations, and {$C^*$}-envelopes},
	J. Funct. Anal. \textbf{158} (1998), no.~2, 389-457.
	
	
	\bibitem{MS99}
	P.~S. Muhly and B. Solel, \emph{Tensor Algebras, {I}nduced Representations,
		and the {W}old Decomposition}, Canad. J. Math \textbf{51} (1999), no.~4, 850-880.
	
	\bibitem{N29}
	J.v. Neumann, \textit{Zur Algebra der Funktionaloperationen und Theorie der normalen
		Operatoren}, Math. Ann. \textbf{102} (1930), no.~1, 370-427. 
	
	
	\bibitem{P73}
	W.~L. Paschke, \emph{Inner product modules over {$B^{\ast} $}-algebras},
	Trans. Amer. Math. Soc. \textbf{182} (1973), 443-468.
	
	
	
	\bibitem{P97}
	M. V. Pimsner, \textit{ A class of {$C^*$}-algebras generalizing both
		{C}untz-{K}rieger algebras and crossed products by {${\bf
				Z}$}},
	Free probability theory (Waterloo, ON, 1995), 189-212, Fields Inst. Commun., 12, Amer. Math. Soc., Providence, RI, 1997.
	
	\bibitem{P20}
	G. Popescu, \emph{Doubly $\Lambda$-commuting row isometries, universel models, and classification}, J. Funt. Anal. \textbf{279} (2020), 108798.
	
	\bibitem{P04}
	D. Popovici, \emph{A Wold-type decomposition for commuting isometric pairs}, Proc. Amer. Math. Soc. \textbf{132} (2004), 2303-2314.
	
	\bibitem{RM74}
	M.A Rieffel, \emph{Induced representations of $C^*$-algebras}, Advances in Math. \textbf{13} (1974), 176-257.
	
	\bibitem{SZ08}
	A. Skalski and J. Zacharias, \emph{Wold decomposition for representations of product systems of {$C^*$}-correspondences}, International J. Math. \textbf{19} (2008), no. 4, 455-479.
	
	\bibitem{S06}
	B. Solel, \emph{ Representations of product systems over semigroups and dilations of commuting CP maps},
	J. Funct. Anal. \textbf{180} (2006), no. 2, 593-618.
	
	\bibitem{S08}
	B. Solel, \emph{Regular dilations of representations of product systems},
	Math. Proc. R. Ir. Acad. \textbf{180} (2008), no. 1, 89-110.
	
	
	\bibitem{TV21}
	H. Trivedi and S. Veerabathiran, \emph{Doubly commuting invariant subspaces for representations of
		product systems of $C^*$-correspondences},
	Ann. Funct. Anal. \textbf{12} (2021), no. 3, Paper No. 47, 32.
	
	
	
	\bibitem{W13}
	M. Weber, \emph{On $C^*$-algebras generated by isometries with twisted commutation relations}, J. Funct. Anal. \textbf{264(8)} (2013) 1975-2004.
	
	\bibitem{W38}
	H. Wold, \emph{A study in the analysis of stationary time series}, Almquist and Wiksell, Uppsala, 1938.
	
	\bibitem{Y99}
	R. Yang, \emph{BCL index and Fredholm tuples}, Proc. Amer. Math. Soc. \textbf{127} (1999), 2385-2393.
	
	\bibitem{Y03}
	R. Yang, \emph{A trace formula for isometric pairs}, Proc. Amer. Math. Soc. \textbf{132} (2003), 533-541.
	
	
	\bibitem{Y05}
	R. Yang, \emph{The core operator and congruent submodules}, J. Funct. Anal. \textbf{228} (2005), 469-489.
\end{thebibliography}
\end{document}